\documentclass[reqno]{amsart}

\usepackage{amssymb,latexsym,amsmath,amsthm}

\newtheorem{theorem}{Theorem}
\newtheorem{lemma}[theorem]{Lemma}

\newcommand{\newsection}[1] {\bigskip\section{#1}\setcounter{theorem}{0}\par\noindent}

\renewcommand{\epsilon}{\varepsilon}
\renewcommand{\l}{\lambda}
\renewcommand{\th}{\theta}

\renewcommand{\Im}{\text{Im }}
\newcommand{\g}{{\rm g}}
\newcommand{\RR}{\ensuremath{\mathbb{R}}}
\newcommand{\R}{\ensuremath{\mathbb{R}}}
\newcommand{\CC}{\ensuremath{\mathbb{C}}}

\newcommand{\prtl}{\ensuremath{\partial}}

\newcommand{\hf}{\ensuremath{\frac{1}{2}}}
\newcommand{\supp}{\ensuremath{\text{supp}}}

\newcommand{\veps}{\ensuremath{\varepsilon}}
\newcommand{\cd}{\,\cdot\,}

\begin{document}

\title[Strichartz estimates on manifolds with boundary]{Strichartz estimates and the 
nonlinear Schr\"odinger equation on manifolds with boundary}
\thanks{The authors were supported by National Science
Foundation grants DMS-0801211, DMS-0654415, and DMS-0555162.}

\author{Matthew D.~Blair \and Hart F.~Smith \and Chris~D.~Sogge}

\begin{abstract}
We establish Strichartz estimates for the Schr\"odinger
equation on Riemannian manifolds $(\Omega,\g)$ with boundary, for both
the compact case and the case that $\Omega$ is the exterior
of a smooth, non-trapping obstacle in Euclidean space.
The estimates for exterior domains are scale invariant; the range
of Lebesgue exponents $(p,q)$ for which we obtain these estimates is
smaller than the range known for Euclidean space, but includes
the key $L^4_tL^\infty_x$ estimate, which we use to
give a simple proof of well-posedness results for the energy critical Schr\"odinger
equation in 3 dimensions. Our estimates on compact manifolds involve a loss of derivatives
with respect to the scale invariant index. We use these to establish well-posedness for
finite energy data of certain semilinear Schr\"odinger equations on general compact manifolds with boundary.
\end{abstract}

\maketitle

\newsection{Introduction}

\noindent
Let $(\Omega,\g)$ be a Riemannian manifold with boundary, of dimension 
$n \geq 2$, and let $v(t,x): [0,T] \times \Omega \to \CC$ be the solution
to the Schr\"{o}dinger equation
\begin{equation}\label{schrodeqn}
(i\prtl_t + \Delta_\g)v(t,x) =0\,, \qquad v(0,x)= f(x)\,.
\end{equation}
We assume in addition that $v$ satisfies either Dirichlet or Neumann 
boundary conditions
$$
v(t,x)\big|_{\prtl \Omega} =0 \qquad \text{ or } \qquad \prtl_\nu
v(t,x)\big|_{\prtl \Omega} =0\,,
$$
where $\prtl_\nu$ denotes the normal derivative along the boundary.
In this work, we consider local in time 
Strichartz estimates for such solutions; 
these are a family of space-time integrability estimates of the form
\begin{equation}\label{strichartz}
\|v\|_{L^p([0,T]; L^q(\Omega))} \leq C\|f\|_{H^s(\Omega)}\,.
\end{equation}
Here $H^s(\Omega)$ denotes the $L^2$ Sobolev space of order $s$, defined with 
respect to the spectral resolution of either the Dirichlet or Neumann 
Laplacian.  The Lebesgue exponents will always be taken to satisfy $p,q \geq 2$,
and always the Sobolev index satisfies $s\ge 0$.

The consideration of high frequency bump function solutions
to~\eqref{schrodeqn} shows that $p,q,s$ must satisfy
\begin{equation}\label{scalingineq}
\frac 2p + \frac nq \geq \frac n2 - s\,.
\end{equation}
In the case where equality holds in \eqref{scalingineq} 
the estimate is said to be {\it scale invariant}; 
otherwise, there is said to be a  
\emph{loss of derivatives} in the estimate, as it deviates from the optimal 
regularity predicted by scale invariance.

Strichartz estimates are most well understood over 
Euclidean space, where $\Omega = \RR^n$ and $\g_{ij}=\delta_{ij}$. 
In this case, the scale invariant estimates hold with $s=0$ and $T=\infty$.
See for example Strichartz~\cite{strich77}, Ginibre and Velo~\cite{ginvelo85},
Keel and Tao~\cite{keeltao98}, and references therein.
Scale invariant estimates for $s>0$ then follow by Sobolev embedding;
such estimates will be called {\it subcritical}, as their proof does
not use the full rate of dispersion for the equation~\eqref{schrodeqn}.

This paper is primarily concerned with proving scale invariant 
Strichartz estimates on the domain exterior to a non-trapping
obstacle in $\R^n$, that is,
$\Omega = \RR^n \setminus \mathcal{K}$ for some compact set $\mathcal{K}$ 
with smooth boundary. Non-trapping means that every unit speed broken 
bicharacteristic escapes each compact subset of $\Omega$ in finite time.
While we are only able to prove such estimates for a restricted range
of subcritical $p,q$, we do obtain estimates with applications to
wellposedness in the energy space for semilinear Schr\"odinger equations when $n=3$.

The key new step in this paper is to establish (for the same range of $p,q$)
scale invariant estimates for the semi-classical Schr\"odinger equation
on a general compact Riemannian manifold with boundary.
The step from these local estimates to the case of exterior domains
depends on the local smoothing bounds of Burq, G\'erard, and Tzvetkov 
\cite{bgtexterior}. When $\mathcal{K}$ is assumed to be non-trapping,
they proved that
\begin{equation}\label{locsmooth}
\|\psi v\|_{L^2([0,T]; H^{s+\hf}(\Omega))} \leq C\|f\|_{H^s(\Omega)}\,, \qquad 
\psi \in C_c^\infty (\overline{\Omega})\,, \qquad s \in [0,1]\,.
\end{equation}
This inequality is a natural formulation of the local smoothing estimates for 
Euclidean space which originated in the work of Constantin and 
Saut~\cite{consaut}, Sj\"olin~\cite{Sjolin}, and Vega~\cite{vegasmooth}.  
The estimate~\eqref{locsmooth} was used in \cite{bgtexterior}
to obtain Strichartz estimates with a loss of $1/p$ derivatives,
by combining the gain in regularity in \eqref{locsmooth} 
with Sobolev embedding, in order to prove space-time integrability estimates
near the obstacle.

Improved results were obtained by Anton~\cite{antonext}, which show that 
Strichartz estimates hold with a loss of $\frac{1}{2p}$ derivatives. 
The approach in \cite{antonext} combines the local smoothing estimates
\eqref{locsmooth} with a semi-classical parametrix construction,
rather than Sobolev embedding. We remark that further improvement
is possible by using the parametrix construction of the authors 
in~\cite{BSSpams}, to yield a loss of $\frac{1}{3p}$ derivatives. 
This is currently the best known estimate for critical $p,q$, that is,
$\frac 2p + \frac nq = \frac n2$, except for the case
where $\mathcal{K}$ is strictly convex. For the exterior domain
to a strictly convex $\mathcal{K}$, the full range of Strichartz estimates
(except for endpoints) was obtained by Ivanovici \cite{ivanovici} for
Dirichlet conditions,
using the Melrose-Taylor parametrix construction.

The use of local smoothing to establish Strichartz bounds has origins 
in the work of Journ\'e, Soffer, and Sogge~\cite{jss},
and of Staffilani and Tataru~\cite{ST}.  Both deal with perturbations of the 
flat Laplacian in $\RR^n$, and establish estimates with no loss of 
derivatives.  The paper \cite{jss} considered 
the case of potential terms $-\Delta + V$, 
whereas \cite{ST} considers non-trapping metric perturbations.  In both cases,
one has local smoothing estimates similar to~\eqref{locsmooth}.

More recently,
Planchon and Vega~\cite{planvega} used a bilinear virial identity 
to obtain the scale invariant estimate~\eqref{strichartz} where
$p=q=4$, $s= \frac 14$ in $n=3$ dimensions (along with a range of 
related inequalities), for the Dirichlet problem on non-trapping exterior 
domains. These estimates were applied to semilinear Schr\"odinger equations, 
showing that for defocusing, energy subcritical nonlinearities, one has 
global existence for initial data in $H^1(\Omega)$.
For strictly convex $\mathcal{K}$, the work \cite{ivanovici} 
establishes global existence for the energy critical semilinear equation, 
focusing or defocusing, for small Dirichlet data in $H^1(\Omega)$.

In the present work, we establish the Strichartz estimates
\eqref{strichartz} for a range of subcritical $p,q$. The key
tool is a microlocal parametrix construction previously used for the 
wave equation in~\cite{SmSoBdry} and~\cite{bdrystz}.
This approach treats both Dirichlet and Neumann boundary conditions,
and applies to general non-trapping obstacles.

\begin{theorem}\label{extthm}
Let $\Omega = \RR^n \setminus \mathcal{K}$ be the exterior domain to
a compact non-trapping obstacle with smooth boundary, and
$\Delta$ the standard Laplace operator on $\Omega$, subject to either
Dirichlet or Neumann conditions.
Suppose that $p>2$ and $q<\infty$ 
satisfy
\begin{equation}\label{pqcondn}
\begin{cases}
\frac 3p + \frac 2q \leq 1\,, & n =2\,,\\
\frac 1p + \frac 1q \leq \frac 12\,, & n \geq 3\,.
\end{cases}
\end{equation}
Then for the solution $v=\exp(it\Delta)f$ to the Schr\"{o}dinger
equation \eqref{schrodeqn}, the following estimates hold
\begin{equation}\label{omegastz}
\|v\|_{L^p([0,T]; L^q(\Omega))} \leq C\|f\|_{H^s(\Omega)}\,,
\end{equation}
provided that
\begin{equation}\label{scalingcond}
\frac 2p + \frac nq = \frac n2 - s\,.
\end{equation}
For Dirichlet boundary conditions, the estimates hold with $T=\infty$.
\end{theorem}

That one may take $T=\infty$ in \eqref{omegastz} for 
Dirichlet boundary conditions is a consequence of the fact
that \eqref{locsmooth} holds for $T=\infty$ in the Dirichlet case.

\medskip

We now consider estimates for compact Riemannian manifolds $\Omega$, with 
$\Delta_\g$ the Laplace-Beltrami operator for $\g$. Burq, G\'erard, and Tzvetkov showed
in \cite{burq1} that for $p>2$ estimates hold with a loss of 
$\frac 1p$ derivatives in case $\prtl \Omega = \emptyset$. 
The same result was established for compact manifolds with 
geodesically concave boundary in~\cite{ivanovici}.
For general boundaries, we establish estimates with the same loss of $\frac 1p$ 
derivatives, valid for $(p,q)$ satisfying~\eqref{pqcondn}.  For such $(p,q)$, 
this is an improvement over the estimates of~\cite{BSSpams}, which involve
a loss of $\frac4{3p}$ derivatives.

\begin{theorem}\label{intthm}
Let $\Omega$ be a compact Riemannian manifold with boundary.
Suppose that $p > 2$ and $q<\infty$
satisfy~\eqref{pqcondn}.  Then for the solution 
$v=\exp(it\Delta_\g)f$ to the Schr\"odinger
equation \eqref{schrodeqn}, the following estimates hold for fixed finite $T$
$$
\|v\|_{L^p([0,T]; L^q(\Omega))} \leq C\|f\|_{H^{s+\frac 1p}(\Omega)}
$$
for $p$, $q$, $s$ satisfying~\eqref{scalingcond}.
\end{theorem}

As with \cite{burq1} and \cite[Corollary 1.5]{ivanovici}, the loss of $\frac 1p$ arises
as a consequence of using a representation for solutions that is
valid only in a local coordinate chart; that is, on a semi-classical time
scale.

In the last two sections of this paper we present applications of the above theorems
to well-posedness of semilinear Schr\"odinger equations in three space dimensions 
with finite energy data.
In Section \ref{wpextdom}, we use Theorem \ref{extthm} and interpolation to establish
the $L^4_tL^\infty_x$ Strichartz estimate.
This estimate yields a simple
proof of well-posedness for small energy data to the energy critical equation on
exterior domains, a result first established by Ivanovici and Planchon \cite{IP}.
In Section \ref{bddomain}, we establish a variant in three dimensions of Theorem \ref{intthm}
for the case $p=2$, for data $u$ localized to dyadic frequency scale $\lambda$. The estimate involves
a loss of $(\log\l)^2$ relative to the estimates of \cite{burq1}. Following the Yudovitch
argument as in \cite[Section 3.3]{burq1}, we use this to establish well-posedness
for finite energy data to certain semilinear Schr\"odinger equations, on general three dimensional compact manifolds with boundary. The logarithmic loss in the
estimates restricts our result to slower growth nonlinearities than considered in \cite{burq1}
for manifolds without boundary. For particular three dimensional manifolds without boundary, recent results have been
obtained for the energy critical case by Herr \cite{Herr}, Herr, Tataru, and Tzvetkov \cite{HTT}, 
and Ionescu and Pausader \cite{IoPa}.

The outline of this paper is as follows.  In section 2, we reduce Theorems \ref{extthm}-\ref{intthm} 
to estimates on the unit scale within a single coordinate chart.  
Section 3 outlines the angular localization approach from~\cite{SmSoBdry}, 
and introduces a wave packet parametrix construction.  Estimates on the 
parametrix are then developed in section 4.  We conclude in sections 5 and 6 with the
applications to semilinear Schr\"odinger equations.

Throughout this paper we use the following notation.
The expression $X \lesssim Y$ means
that $X \leq CY$ for some $C$ depending only on the manifold,
metric, and possibly the triple $(p,q,s)$ under
consideration. Also, we abbreviate $L^p(I;L^q(U))$
by $L^p L^q(I \times U)$ or by $L^p_T L^q(U)$ when $I=[0,T]$. If $U=\R^n$ we write $L^p_TL^q$. 
As will be seen below, the last component of an $n$-vector will take on 
special meaning, hence we will often write $x=(x',x_n)$ so that $x'$ 
denotes the first $n-1$ components.

We conclude this introduction with a remark on the Sobolev spaces
that we use in the case of exterior domains. In the above theorems,
the Sobolev space $H^s(\Omega)$ and the
operator $\exp(it\Delta)$ are defined using the spectral
resolution of $\Delta$ subject to the chosen Dirichlet or Neumann
boundary condition $B$; in particular, the linear evolution
preserves $H^s(\Omega)$. The space $H^2(\Omega)$ is then
equal to the subspace of
$H^2(\overline\Omega)$ satisfying $Bu=0$, and for $0\le s\le 2$, the space
$H^s(\Omega)$
can be defined by interpolation. For $s\ge 2$, these spaces satisfy
$u\in H^s(\Omega)$ if and only if $Bu=0$ and $\Delta u\in H^{s-2}(\Omega)$.
Thus, for large values of $s$, a function in $H^s(\Omega)$
satisfies the linear compatibility conditions $B(\Delta^k u)=0$,
for $k$ for which this is defined. These compatibility conditions are
necessary to bootstrap the local smoothing estimates \eqref{locsmooth}
to higher orders $s$, as well as to insure $v(t,\cd)\in H^s(\Omega)$,
which is required to handle commutator terms with cutoff functions.
We will also use that
$$
\|v\|_{H^s(\Omega)}\approx 
\|\psi v\|_{H^s(\tilde\Omega)}+\|(1-\psi)v\|_{H^s(\R^n)}\,,
$$
where $\psi\in C_c^\infty(\overline\Omega)$ is such that $1-\psi$ vanishes
on a neighborhood of $\partial\Omega$, and $\tilde\Omega$ is a compact
manifold with boundary in which $\Omega\cap\supp(\psi)$ embeds isometrically;
for example $\tilde\Omega=\Omega\cap [-R,R\,]^n$ 
with periodic boundary conditions and $R$ sufficiently large.

We use $H^s(\overline\Omega)$ to denote the space of
extendable elements, with no
boundary conditions. For $\Omega$ an exterior domain, $H^s(\overline\Omega)$
consists of
restrictions of functions in $H^s(\R^n)$ to $\Omega$ 
with the quotient norm (minimal norm of an extension); 
for $\Omega$ compact we embed $\Omega$ in a compact
manifold $\Omega'$ without boundary, and
$H^s(\overline\Omega)$ consists of restrictions of elements $H^s(\Omega')$.
By elliptic regularity, $H^s(\Omega)\subset H^s(\overline\Omega)$.

\vfill
\pagebreak

%%%%%%%%%%%%%%%%%%%%%%%%%%%%%%%%%%%%%%%%%%%%%%%%%%%

\newsection{Preliminary reductions}

\noindent
In this section, we reduce the inequalities in
Theorems~\ref{extthm} and~\ref{intthm} to estimates on solutions
to a pseudodifferential equation defined in a
coordinate chart near the boundary. We start by considering the case
of Theorem~\ref{extthm}.

For $\Omega=\R^n \setminus \mathcal{K}$,
we take $\psi\in C_c^\infty(\overline\Omega)$ such that 
$1-\psi$ vanishes on a neighborhood of $\partial\Omega=\partial\mathcal{K}$. 
Then
$v_0=(1-\psi)v$ satisfies the inhomogeneous Schr\"odinger equation on $\R^n$:
$$
\bigl(i\partial_t+\Delta\bigr)v_0=[\psi,\Delta]v\,,
\qquad\quad
v_0|_{t=0}=(1-\psi)f\,.
$$
Here, $(1-\psi)f\in H^s(\R^n)$, and by \eqref{locsmooth} we have
$[\Delta,\psi]v\in L^2_TH^{s-\frac 12}(\R^n)$. (Although stated
only for $s\in[0,1]$ in \cite{bgtexterior}, it is easy to see
by a bootstrap argument and interpolation
that \eqref{locsmooth} holds for all $s\ge 0$, where $H^s$ is the
is the intrinsic Sobolev space for the Dirichlet/Neumann conditions
as above.)
The Strichartz estimates for $v_0$ then follow from Proposition 2.10 of
\cite{bgtexterior} together with Sobolev embedding.
While \cite{bgtexterior} considers the case $s\in [0,1]$, the result
of Proposition 2.10 of \cite{bgtexterior}
follows for all $s>0$, since the free
Schr\"odinger propagator $\exp(it\Delta)$ on $\R^n$ commutes with differentiation.

We are thus reduced to establishing estimates on the term $\psi v$.
We isometrically embed a neighborhood of $\supp(\psi)$ into a compact manifold
$(\tilde\Omega,\g)$ with boundary, where $\partial\tilde\Omega=\partial\Omega$. 
Then $v_1=\psi v$ satisfies the inhomogeneous
Schr\"odinger equation on $\tilde\Omega$:
$$
\bigl(i\partial_t+\Delta_\g\bigr)v_1=[\Delta,\psi]v\,,
\qquad\quad
v_1|_{t=0}=\psi f\,.
$$
By \eqref{locsmooth}, we are reduced to establishing the following
estimate over a compact manifold with boundary $\tilde\Omega$,
\begin{equation}\label{squareest}
\|v\|_{L^p_TL^q(\tilde\Omega)}\lesssim
\|v\|_{L^2_TH^{s+\frac 12}(\tilde\Omega)}+
\|(i\partial_t+\Delta_\g)v\|_{L^2_TH^{s-\frac 12}(\tilde\Omega)}\,.
\end{equation}
Here we use that $[\Delta,\psi]$ vanishes near $\partial\Omega$,
hence maps $H^{s+\frac 12}(\Tilde\Omega)\rightarrow H^{s-\frac 12}(\Tilde\Omega)$.

We next take a Littlewood-Paley decomposition of $v$ in the $x$ variable
with respect to the spectrum for $\Delta_\g$. Precisely, we write
$$
v=\beta_0(-\Delta_\g)+\sum_{j=1}^\infty \beta\bigl(2^{-2j}(-\Delta_\g)\,\bigr)v
\equiv \sum_{j=0}^\infty v_j\,,
$$
where $\sum_{j=1}^\infty \beta(2^{-2j}s)=1$ for $s\ge 2$, and
$\beta$ is supported by $s\in[\frac 12, 2]$. The low frequency
terms are easily dealt with by Sobolev embedding, since the right hand
side of \eqref{squareest} 
controls the $L^p_TH^{s-\frac 12}$ norm of $v$ for $2\le p\le\infty$.
The following square function estimate holds, for example by heat kernel
methods,
\begin{equation}\label{sqfn}
\|v\|_{L^p_TL^q(\tilde\Omega)}\approx
\bigl\|\bigl(\,\textstyle\sum_j|v_j|^2\bigr)^\hf\bigr\|_{L^p_TL^q(\tilde\Omega)}
\le
\Bigl(\sum_j\|v_j\|_{L^p_TL^q(\tilde\Omega)}^2\Bigr)^\hf\,,
\end{equation}
where we use $p,q\ge 2$ in the last step.
By orthogonality, the desired estimate
\eqref{squareest} would then follow
as a consequence of the following estimate,
\begin{equation*}
\|v_j\|_{L^p_TL^q(\tilde\Omega)}\lesssim
2^{j(s+\frac 12)}\|v_j\|_{L^2_T L^2(\tilde\Omega)}+
2^{j(s-\frac 12)}\|(i\partial_t+\Delta_\g)v_j\|_{L^2_T L^2(\tilde\Omega)}\,.
\end{equation*}
Finally, we divide $[0,T]$ into intervals of length $2^{-j}$ and note
that, since $p,q\ge 2$, by the Minkowski inequality
it suffices to prove the above on each subinterval; that is, for $T=2^{-j}$.
To summarize, Theorem \ref{extthm} is thus reduced to establishing
the following semiclassical result.

\begin{theorem}\label{semiclassthm}
Let $\Omega$ be a compact Riemannian manifold with boundary, and
$\Delta_\g$ the Laplace-Beltrami operator, 
subject to either Dirichlet or Neumann boundary conditions.
Suppose that $p > 2$ and $q<\infty$ satisfy~\eqref{pqcondn} 
and~\eqref{scalingcond}.
 
Suppose also that, for all $t$, $v_\l(t,\cd)$ is spectrally localized
for $-\Delta_\g$ to the range $[\frac 14 \l^2,4\l^2]$.
Then the following estimate holds, uniformly over $\l$,
\begin{equation}\label{squarelest}
\|v_\l\|_{L^p_{\l^{-1}}L^q(\Omega)}\lesssim
\l^s\Bigl(\,\l^{\hf}\|v_\l\|_{L^2_{\l^{-1}} L^2(\Omega)}+
\l^{-\hf}\|(i\partial_t+\Delta_\g)v_\l\|_{L^2_{\l^{-1}} L^2(\Omega)}\Bigr)\,.
\end{equation}
\end{theorem}

\medskip

We observe that Theorem \ref{intthm} also follows as a consequence
of \eqref{squarelest}. To see this, we divide $[0,T]$ into 
subintervals of length $\l^{-1}$, and note that for 
$v_\l=\exp(it\Delta_\g)f_\l$,
on each subinterval the right hand side of \eqref{squarelest}
is bounded by $\l^s\|f_\l\|_{L^2(\Omega)}$. Summing the $L^p$ norm over 
a total of $\approx \l$ subintervals leads to
$$
\|v_\l\|_{L^p_TL^q(\Omega)}\lesssim
\l^{s+\frac 1p}\|f_\l\|_{L^2(\Omega)}\approx\|f_\l\|_{H^{s+\frac 1p}(\Omega)}\,.
$$
Applying the square function estimate \eqref{sqfn} as above 
yields Theorem \ref{intthm}.

We will establish \eqref{squarelest} by the methods developed in \cite{SmSoBdry}
and \cite{bdrystz} to obtain dispersive estimates for the wave
equation on manifolds with boundary.
We start by
taking a finite partition of unity over $\Omega$, subordinate to a cover
by coordinate patches. We restrict attention to a coordinate patch centered
on $\partial\Omega$; the interior terms can be handled by the methods
of \cite{burq1}, or by the parametrix construction of this paper.
Thus, let $\psi\in C^\infty_c(\overline\Omega)$ be supported in a 
boundary normal coordinate patch along $\partial\Omega$.
The function $\psi v_\l$ is not sharply spectrally localized,
but does remain spectrally concentrated in
frequencies $\le\l$. Precisely, for all $k\ge 0$,
\begin{equation}\label{specloc}
\|\psi v_\l\|_{L^2_{\l^{-1}}H^k(\overline\Omega)}
\lesssim\|v_\l\|_{L^2_{\l^{-1}}H^k(\Omega)}
\lesssim \l^k\|v_\l\|_{L^2_{\l^{-1}}L^2(\Omega)}\,,
\end{equation}
and the same holds with $v_\l$ replaced by
$(i\partial_t+\Delta_\g)v_\l$.

Letting $x_n$ denote geodesic distance to the boundary, and $x'$ coordinates
on $\partial\Omega$, in boundary normal coordinates the Laplace operator takes the form
$$
\Delta_g v = \rho^{-1}\sum_{1 \leq i,j \leq n}\prtl_i \left(
\,\g^{ij} \rho \, \prtl_j v\, \right)
$$
where $\rho = \sqrt{\det \g_{lk}}$ and $\g^{ij}$ denotes the
inverse of the metric $\g_{lk}$. Furthermore, $\g^{in} = \g^{ni} = \delta_{in}$,
so there are no mixed $\partial_{x'}\partial_{x_n}$ terms.

We now extend the metric $\g(x',x_n)$ in an even manner across $x_n=0$;
the new metric $\g(x',|x_n|)$, which we also denote by $\g$,
is defined on an open subset of $\R^n$,
and is of Lipschitz regularity. We extend the solution $\psi v_\l$
in an odd or even fashion, corresponding to Dirichlet or Neumann boundary
conditions, to obtain a $C^{1,1}$ function. We will assume $\psi$
is chosen so that $\psi(x',x_n)$ is independent of $x_n$ near $x_n=0$.
Since the extended Laplace operator is even, the regularity of $\g$ and
$v_\l$ show that the extended solution satisfies the extended equation 
across $x_n=0$,
$$
(i\prtl_t + \Delta_\g)(\psi v_\l)=[\Delta_\g,\psi]v_\l
+\psi (i\prtl_t + \Delta_\g)v_\l\,,
$$
where $\Delta_\g v_\l$ is extended oddly/evenly
as is $v_\l$, and $\psi$ is even.

By choosing sufficiently small coordinate patches, and rescaling if
necessary, we may assume that $\g$ extends to all of $\R^n$, such that
$$
\|\g^{ij}-\delta^{ij}\|_{C^{0,1}(\R^n)}\le c_0\ll 1\,,
\qquad\quad
\g^{ij}=\delta^{ij}\;\;\text{if}\;\;|x|>1\,.
$$

The odd (respectively even) extension operator maps functions in
$H^r(\overline{\R^n_+})$ satisfying $f(x',0)=0$ 
(respectively $\partial_{x_n}f(x',0)=0$ ) to functions in $H^r(\R^n)$,
provided $r\in [0,\frac 52)$. The extension also commutes with differentiation
in the $x'$ variables. We observe that multiplication by functions
such as $\g$ or $\rho$ preserves $H^r(\R^n)$ for $r\in [0,\frac 32)$, and
multiplication by $\partial_x\rho$ preserves 
$H^r(\R^n)$ for $r\in [0,\frac 12)$. This can be seen, e.g., from the
fact that $\langle \xi\rangle^{\frac 12-\epsilon}$ is an $A_2$ weight
in one dimension, and that $\partial_x\rho$ is a Calder\'on-Zygmund type
multiplier in $x_n$.

It follows that the bound \eqref{specloc} holds to a limited extent for the
extension of $\psi v_\l$ to $\R^n$.
To quantify this, we introduce the following family of norms, for $r\ge 0$,
$$
\|f\|_{H^{r,\l}}=
\sum_{|\alpha|\le N}\Bigl(\l^{-|\alpha|}\|\partial^\alpha_{x'}f\|_{L^2(\R^n)}
+\l^{-|\alpha|-r}\|\partial^\alpha_{x'}f\|_{H^r(\R^n)}\Bigr)\,,
$$
and observe that $\|f\|_{H^{\sigma,\l}}\lesssim \|f\|_{H^{r,\l}}$ if $0\le\sigma\le r$.
Here $N$ is taken to be a fixed but sufficiently large number,
which we allow to change in a given inequality. 
However, for the results of this paper $N$ need never exceed $n+2$.

By \eqref{specloc} and the
above, it holds that for $2\le r <\frac 52$
\begin{multline}\label{tangest}
\l^{\frac 12}\|\psi v_\l\|_{L^2_{\l^{-1}}H^{r,\l}(\R^n)}+
\l^{-\frac 12}\|(i\partial_t+\Delta_\g)(\psi v_\l)\|_{L^2_{\l^{-1}}H^{r-2,\l}(\R^n)}
\\
\lesssim
\l^{\frac 12}\|v_\l\|_{L^2_{\l^{-1}}L^2(\Omega)}+
\l^{-\frac 12}\|(i\partial_t+\Delta_\g)v_\l\|_{L^2_{\l^{-1}}L^2(\Omega)}\,.
\end{multline}
This bound also holds if we replace $\Delta_\g$ on the left side by the
divergence form operator $\partial_i\g^{ij}\partial_j$, since
the difference $\rho^{-1}(\partial_i\rho)\g^{ij}\partial_j$ maps
$H^{r,\l}\rightarrow H^{r-2,\l}$ with norm $\l$, provided
$r\in [2,\frac 52)$. Since subsequent
estimates will be only in terms of the left hand side of \eqref{tangest},
we may thus set $\rho \equiv 1$, and replace $\Delta_\g$ by $\sum_{ij}\partial_i\,\g^{ij}\partial_j$.

We next reduce matters to considering solutions that are 
strictly frequency localized on $\R^n$, and which satisfy
an equation with frequency localized coefficients.
For each $\mu$, we form regularized coefficients $\g^{ij}_\mu$ 
by truncating the 
Fourier transform of the $\g^{ij}$ so that
\begin{equation}\label{trunc}
\supp(\widehat{\g^{ij}_\mu}) \subset \{|\xi| \leq c \mu\}\,,
\end{equation}
for some small constant $c$. We observe the following estimates
\begin{equation}\label{regest}
\|\g^{ij}_\mu - \g^{ij}\|_{L^\infty(\R^n)} \lesssim \mu^{-1}, \qquad\quad
\|\prtl_x^\alpha \g^{ij}_\mu\|_{L^\infty(\R^n)} \lesssim \mu^{|\alpha|-1}\,,
\quad|\alpha|\ge 1\,.
\end{equation}
With slight abuse of notation we now set
$$
\Delta_{\g_\mu} v = \sum_{1 \leq i,j \leq n}\prtl_i \left(
\,\g^{ij}_\mu \, \prtl_j v\, \right)\,.
$$
We will prove in the next section
the following estimate for $u_\mu(t,x)$ defined
on $[0,\mu^{-1}]\times\R^n$, which are localized to spatial frequencies
$\approx \mu$.

\begin{lemma}\label{coniclemma}
Suppose that $(p,q,s)$ are as in Theorem \ref{extthm},
and $\widehat u_\mu(t,\xi)$ is supported in the region
$\frac 12\mu\le |\xi|\le \frac 52\mu$. Then
\begin{equation}\label{conicest1}
\|u_\mu\|_{L^p_{\mu^{-1}}L^q}\lesssim 
\mu^{s+\hf}\|u_\mu\|_{L^2_{\mu^{-1}}L^2}+
\mu^s\|(i\partial_t+\Delta_{\g_\mu})u_\mu\|_{L^1_{\mu^{-1}}L^2}\,.
\end{equation}
Furthermore, if $\widehat u_\mu(t,\xi)$ is in addition localized to 
$|\xi'|\le \frac 32 |\xi_n|$, then \eqref{conicest1}
holds for $p>2$ and $q<\infty$ satisfying \eqref{scalingcond} with $s\ge 0$; 
that is, without the restriction \eqref{pqcondn}.
\end{lemma}

In the remainder of this section we reduce 
\eqref{squarelest}, and hence
Theorems \ref{extthm} and \ref{intthm}, to establishing Lemma \ref{coniclemma}.

We start by considering the frequency components $\mu\le\l$
of $\psi v_\l$ (by which we understand its odd/even extension to $\R^n$). 
Let $\beta_\mu(D)$
denote a Littlewood-Paley localization operator on $\R^n$
to frequencies $\approx\mu$,
and consider $u_\mu=\beta_\mu(D)(\psi v_\l)$.
For $\mu\le\l$, \eqref{conicest1} and the Schwarz inequality imply
\begin{equation}\label{lowfreqest}
\|u_\mu\|_{L^p_{\l^{-1}}L^q}\lesssim 
\mu^s\Bigl(\,\l^\hf\|u_\mu\|_{L^2_{\l^{-1}}L^2}+
\l^{-\hf}\|(i\partial_t+\Delta_{\g_\mu})u_\mu\|_{L^2_{\l^{-1}}L^2}\Bigr)\,.
\end{equation}
Since $s\ge 0$ in our bounds, we may sum over dyadic values of
$\mu\le\l$ to establish \eqref{squarelest} for the cutoff of $\psi v_\l$
to frequencies $\le\l$, provided we bound the $\ell^2$ norm over $\mu$ of the terms in parentheses
in \eqref{lowfreqest} by
the terms in parentheses in \eqref{squarelest}.
By \eqref{tangest}, this is a special case
of the following estimate, which we establish for all $r\in [2,\frac 52)$,
\begin{multline}\label{umuest}
\biggl(\sum_\mu \;\l\|u_\mu\|^2_{L^2_{\l^{-1}}H^{r,\l}}+
\l^{-1}\|(i\partial_t+\Delta_{\g_\mu})u_\mu\|^2_{L^2_{\l^{-1}}H^{r-2,\l}}\biggl)^\hf\\
\lesssim
\l^{\frac 12}\|\psi v_\l\|_{L^2_{\l^{-1}}H^{r,\l}}+
\l^{-\frac 12}\|(i\partial_t+\Delta_\g)(\psi v_\l)\|_{L^2_{\l^{-1}}H^{r-2,\l}}\,.
\end{multline}
Since $\beta_\mu$ is $L^2$ bounded and commutes with differentiation, 
this will follow from showing the fixed time estimate
\begin{equation}\label{umuest'}
\biggl(\sum_\mu\;\|(\beta_\mu\Delta_\g-\Delta_{\g_\mu}\beta_\mu)(\psi v_\l)
\|^2_{H^{r-2,\l}}\biggr)^\hf\lesssim
\l\,\|\psi v_\l\|_{H^{r-1,\l}}\,.
\end{equation}
In this estimate we may replace $\Delta_\g=\partial_i\g^{ij}\partial_j$
by $\g^{ij}\partial_i\partial_j$, and similarly for $\Delta_{\g_\mu}$.
This follows since the difference $(\partial_i\g^{ij})\partial_j$
maps $H^{r-1,\l}\rightarrow H^{r-2,\l}$ with norm $\l$.
By the Coifman-Meyer commutator estimate (see~\cite[Prop 4.1D]{Tay}), for $\sigma\in[0,r-2]$,
$$
\biggl(\sum_\mu\|\,[\beta_\mu,\g]\partial_x^2(\psi v_\l)\|^2_{H^\sigma}\biggr)^\hf
\lesssim \|\psi v_\l\|_{H^{\sigma+1}}
\le \l^{\sigma+1}\,\|\psi v_\l\|_{H^{r-1,\l}}\,.
$$
The same holds with $[\beta_\mu,\Delta_\g]$ replaced by 
$[\partial_{x'},[\beta_\mu,\Delta_\g]]$, since this has the effect of
differentiating the coefficients $\g^{ij}$ in $x'$, which remain Lipschitz. 
Hence
$$
\biggl(\sum_\mu\|[\beta_\mu,\g]\partial_x^2(\psi v_\l)\|^2_{H^{r-2,\l}}\biggr)^\hf
\lesssim \l\,\|\psi v_\l\|_{H^{r-1,\l}}\,.
$$
Next, using \eqref{regest} and interpolation, we obtain for $0\le\sigma\le 1$,
$$
\biggl(\sum_\mu\|(\g-\g_\mu)\partial_x^2\beta_\mu(\psi v_\l)\|^2_{H^\sigma}\biggr)^\hf\lesssim
\|\partial_x(\psi v_\l)\|_{H^\sigma}
\le \l^{\sigma+1}\,\|\psi v_\l\|_{H^{\sigma+1,\l}}\,.
$$
Commuting with $\partial_{x'}$ as above yields
$$
\biggl(\sum_\mu\|(\g-\g_\mu)\partial_x^2\beta_\mu(\psi v_\l)\|^2_{H^{r-2,\l}}\biggr)^\hf\lesssim
\l\,\|\psi v_\l\|_{H^{r-1,\l}}\,,
$$
completing the proof of \eqref{umuest'}.

To handle frequencies $\mu>\l$, we consider separately the tangential and
normal components of $u_\mu$. Thus, we decompose
$$
\beta_\mu(\xi)=\Gamma_\mu(\xi)+\Gamma'_\mu(\xi)\,,
$$
where
$$
\text{supp}(\Gamma_\mu)\subset\{\xi\,:\,|\xi'|\le \tfrac 32|\xi_n|\}\,,
\qquad\quad
\text{supp}(\Gamma'_\mu)\subset\{\xi\,:\,|\xi'|\ge |\xi_n|\}\,.
$$

First consider a tangential component $u_\mu=\Gamma'_\mu(\psi v_\l)$.
The key idea is that $u_\mu$ and $\Delta_{\g_\mu}u_\mu$ are
supported where $|\xi'|\approx \mu$,
whereas $\partial_{x'}$ weighs as $\lambda$ acting on $u_\mu$.
Consequently, for each fixed time,
$$
\|u_\mu\|_{L^2}\lesssim \Bigl(1+\frac \mu\l\Bigr)^{-N}\|u_\mu\|_{H^{0,\l}}\,,
$$
and similarly for $(i\partial_t+\Delta_{\g_\mu})u_\mu$.
Noting that the proof of \eqref{umuest} works the same with
with $\beta_\mu$ replaced by $\Gamma'_\mu$, 
then \eqref{conicest1}, \eqref{tangest}, and \eqref{umuest} together yield
$$
\|u_\mu\|_{L^p_{\mu^{-1}}L^q}\lesssim
\Bigl(\frac \mu \l\Bigr)^{-2}\l^s
\Bigl(\l^\hf\|v_\l\|_{L^2_{\l^{-1}}L^2(\Omega)}+
\l^{-\hf}\|(i\partial_t+\Delta_\g)v_\l
\|_{L^2_{\l^{-1}}L^2(\Omega)}\Bigr)\,.
$$
We sum this over the $\mu/\l$ disjoint intervals of length $\mu^{-1}$ contained
in $[0,\l^{-1}]$, and over dyadic values $\mu\ge \l$, to complete the proof of
\eqref{squarelest} for $\psi v_\l$ localized to
frequencies $|\xi'|\ge|\xi_n|$.

For a normal component $u_\mu=\Gamma_\mu(\psi v_\l)$, we do not
have sufficient decay in powers of $\mu/\l$ to apply the above
steps for large $s$. Instead, we use the fact that \eqref{conicest1}
holds for all $s\ge 0$ in this case, and deduce large $s$ results
from the case $s=0$ together with Sobolev embedding. Given a triple
$(p,q,s)$ satisfying \eqref{scalingcond}, 
let $q_0\in [2,\frac ns)$ be such that
$$
\frac 1{q_0}-\frac 1 q=\frac sn\,,
$$
hence $(p,q_0,0)$ satisfies \eqref{scalingcond}.
Then \eqref{conicest1} implies
\begin{equation*}
\|u_\mu\|_{L^p_{\mu^{-1}} L^{q_0}}
\lesssim
\mu^{\hf}\|u_\mu\|_{L^2_{\mu^{-1}} L^2}+
\|(i\partial_t+\Delta_{\g_\mu})u_\mu\|_{L^1_{\mu^{-1}} L^2}
\,.
\end{equation*}
Summing over intervals yields, for $r\in [2,\frac 52)$, and $\mu>\l$,
\begin{align*}
\|u_\mu\|_{L^p_{\l^{-1}} L^{q_0}}
&\lesssim
\frac \mu{\l^\hf}\,\|u_\mu\|_{L^2_{\l^{-1}} L^2}+
\|(i\partial_t+\Delta_{\g_\mu})u_\mu\|_{L^1_{\l^{-1}} L^2}\\
{}
&\lesssim
\Bigl(\frac \mu \l\Bigr)^{2-r}
\Bigl(\l^\hf\|u_\mu\|_{L^2_{\l^{-1}} H^{r,\l}}+
\|(i\partial_t+\Delta_{\g_\mu})u_\mu\|_{L^1_{\l^{-1}} H^{r-2,\l}}\Bigr)\,.
\end{align*}
We apply this inequality to $\l^{-1}\partial_{x'}u_\mu$, and
observe that
$$
\|[\l^{-1}\partial_{x'},\Delta_{\g_\mu}]u_\mu\|_{L^1_{\l^{-1}}H^{r-2,\l}}
\lesssim
\l\,\|u_\mu\|_{L^1_{\l^{-1}}H^{r,\l}}
\lesssim
\l^{\frac 12}\|u_\mu\|_{L^2_{\l^{-1}}H^{r,\l}}\,.
$$
This holds since multiplication by the tangential derivative
$\partial_{x'}\g_\mu$
preserves $H^{r-1,\l}$, provided $r<\frac 52$.
Repeated application yields
\begin{equation*}
\|(\l^{-1}\partial_{x'})^\alpha u_\mu\|_{L^p_{\l^{-1}} L^{q_0}}
\lesssim
\Bigl(\frac \mu \l\Bigr)^{2-r}
\Bigl(\l^\hf\|u_\mu\|_{L^2_{\l^{-1}} H^{r,\l}}+
\|(i\partial_t+\Delta_{\g_\mu})u_\mu\|_{L^1_{\l^{-1}} H^{r-2,\l}}\Bigr)\,.
\end{equation*}
We next apply Sobolev embedding to yield
\begin{align*}
\|u_\mu\|_{L^p_{\l^{-1}} L^q}
&\lesssim
\mu^{\frac sn}\|\,|\partial_{x'}|^{\frac{s(n-1)}n}u_\mu\|_{L^p_{\l^{-1}} L^{q_0}}\\
{}&\lesssim
\mu^{\frac sn}\l^{\frac{s(n-1)}n}\sup_{|\alpha|\le n}
\|(\l^{-1}\partial_{x'})^\alpha u_\mu\|_{L^p_{\l^{-1}} L^{q_0}}\\
{}&\lesssim
\l^s\Bigl(\frac \mu \l\Bigr)^{2+\frac sn-r}
\Bigl(\l^\hf\|u_\mu\|_{L^2_{\l^{-1}} H^{r,\l}}+
\|(i\partial_t+\Delta_{\g_\mu})u_\mu\|_{L^1_{\l^{-1}} H^{r-2,\l}}\Bigr)\,.
\end{align*}
We choose $r\in (2+\frac sn,\frac 52)$, and apply \eqref{umuest}.
We then sum over dyadic values of $\mu>\l$ to establish
\eqref{squarelest} for $\psi v_\l$ localized to
frequencies $|\xi'|\le \frac 32|\xi_n|$.

In proving Lemma \ref{coniclemma} using the results of \cite{SmSoBdry},
it is convenient to work as in that paper with a first order equation.
To do so, we start by rescaling the time interval of length
$\mu^{-1}$ in \eqref{conicest1} to an interval of length 1,
by considering the function $v(t,x)=u_\mu(\mu^{-1}t,x)$.
This replaces $\Delta_{\g_\mu}$ by 
$\mu^{-1}\sum\partial_i(\g^{ij}_\mu\partial_j v)$, which is a first
order symbol for $|\xi|\approx\mu$. 

We can modify this operator away from the region $|\xi|\approx\mu$
without changing the estimate \eqref{conicest1}. To fit into the framework of
\cite{SmSoBdry}, we want to work with an operator such that solutions
to the homogeneous flow remain frequency localized to 
$|\xi|\approx\mu$ if the initial data is supported there.
For $\beta_\mu(\xi)$ a Littlewood-Paley
cutoff to frequencies $\approx\mu$, we thus set
$$
P_\mu(x,D)v =
\mu^{-1}\beta_\mu(D)\partial_i\bigl(\g^{ij}_\mu\partial_j\beta_\mu(D)v\bigr)
+\mu^{-1}(1-\beta_\mu(D)^2)\Delta v\,.
$$
Then $P_\mu$ is an elliptic self-adjoint operator on $\R^n$, with a symbol $p_\mu(x,\xi)$ 
such that
$$
\|\partial^2_{\xi_i\xi_j}p_\mu(x,\xi)-\mu^{-1}I\|\le c_0\ll 1\,.
$$

We will prove in the next sections the following result. Here we replace
the parameter $\mu$ by $\l$, and $v$ by $u$, 
to follow the notation of \cite[Theorem 2.2]{SmSoBdry}.

\begin{theorem}\label{dyadicthm}
Let $\g^{ij}_\l$ be obtained by truncating $\g^{ij}(x',|x_n|)$
to frequencies $\le c\l$, and define $P_\l$ as above.
Suppose that $u_\l(t,x)$ is localized to spatial frequencies
$|\xi|\in[\frac 12 \l,\frac 52 \l]$, and $p,q,s$ satisfy
the conditions of Theorem \ref{extthm}.
Then for small $\veps$ the following holds
\begin{equation}\label{semiclass1}
\|u_\l\|_{L^p_\veps L^q(\R^n)}\lesssim
\l^{s+\frac 1p}\Bigl(\,\|u_\l\|_{L^\infty_\veps L^2(\R^n)}+
\|(i\partial_t+P_\l)u_\l\|_{L^1_\veps L^2(\R^n)}\Bigr)\,.
\end{equation}
If in addition $\hat{u}_\l$ is localized to $|\xi_n|\ge\frac 1{20} |\xi|$, then 
\eqref{semiclass1} holds if $s\ge 0,$ $p>2,$ $q<\infty,$
and $\frac 2p+\frac nq=\frac n2-s.$
\end{theorem}

That this implies Lemma \ref{coniclemma} follows from the fact that
$$
\|u_\l\|_{L^\infty_\veps L^2(\R^n)}\lesssim
\|u_\l\|_{L^2_\veps L^2(\R^n)}+
\|(i\partial_t+P_\l)u_\l\|_{L^1_\veps L^2(\R^n)}\,,
$$
which follows from self-adjointness of $P_\l$.
We further note that, by the Duhamel principle, it suffices to prove
the estimate for the case
that $(i\partial_t+P_\l)u_\l=0$. In particular, it would
suffice to prove \eqref{semiclass1} with $L^1_\veps$ replaced by
$L^2_\veps$ on the right hand side, as is the case
in \cite[Theorem 2.2]{SmSoBdry}.

%%%%%%%%%%%%%%%%%%%%%%%%%%%%%%%%%%%%%%%%%%%%%%%%%%%

\newsection{Angular Localization}

\noindent
We now proceed as in~\cite[\S 3]{SmSoBdry}, and decompose $u_\l$
in the frequency domain into terms localized to angular sectors.
This is done
by taking a finite dyadic decomposition in the $\xi_n$ variable,
where $\xi=(\xi',\xi_n)$.
Precisely, we write $u_\l = \sum_{j=1}^{N_\l} u_j$, 
where $N_\l = \frac 13 \log_2 \l\,,$ and where for $1< j < N_\l$,
\begin{equation}\label{nontgtloc}
\supp\left(\widehat{u}_j(t,\xi)\right) \subset
\bigl\{\,|\xi'| \in [\tfrac 14\l,4\l]\,, \;|\xi_n| \in
[2^{-j-2}\l,2^{-j+1}\l] \, \bigr\}\,.
\end{equation}
The ``tangential'' term is $u_{N_\l}$, where
\begin{equation}\label{tgtloc}
\supp\left(\widehat{u}_{N_\l}(t,\xi)\right) \subset
\bigl\{\,|\xi'| \in [\tfrac 14\l,4\l]\,, \;|\xi_n| \le
\l^{\frac 23} \bigr\}\,,
\end{equation}
and $u_1$ is localized away from tangential co-directions
\begin{equation}\label{angleoneloc}
\supp\left(\widehat{u}_1(t,\xi)\right) \subset
\bigl\{\,|\xi| \in [\tfrac 12 \l,\tfrac 52\l]\,, 
\;|\xi_n| \geq \tfrac 18\l \bigr\}\,.
\end{equation}
The term $u_1$ can be handled by the same methods as $u_2$, 
so we restrict attention to $2\le j \le N_\l$.

The energy of the solution $u_\l$ travels along bicharacteristic curves
$$
\dot x = d_\xi p_\l(x,\xi)\,, \qquad
\dot \xi = - d_xp_\l(x,\xi)\,.
$$
For curves which satisfy $|\xi(t)|\approx \l$ for $t \in
[0,\veps]$, then $|\dot x|\approx 1$, and $\dot x_n\approx\l^{-1}\xi_n$.  
In addition, if $|\xi_n (0)| \approx 2^{-j}\l$, then we will
have $|\xi_n(t)| \approx 2^{-j}\l$ for $|t| \le \veps 2^{-j}$.
Setting $\theta_j=2^{-j}$, the function $u_j$ is thus localized to 
bicharacteristics which remain at angle $\approx\theta_j$ to the
boundary for times up to $\veps \theta_j$.
As in \cite{SmSoBdry}, we will have good estimates on
the term $u_j$ on slabs of width $\veps 2^{-j}$ in $t$. 

For each $j \ge 2 $, we let $p_j(x,\xi)$ be the
regularization of the symbol $p_\l(x,\xi)$,
obtained by truncating the metric coefficients $\g^{ij}_\l(x)$ to
frequencies $\leq c \l^{\frac 12} \theta_j^{-\frac 12}$.

Let $S_{j,k} $ denote the slab $\{(x,t): t \in [k \veps \theta_j,
(k+1)\veps \theta_j]\}$.  The slab $S= [0,\veps]\times \RR^n$
is then the union of the $S_{j,k}$ as $k$ ranges over the
integers from
$0 $ to $\theta_j^{-1}$. 
Define
\begin{equation*}
\sigma(q)=
\begin{cases}
\frac 23(\frac 12-\frac 1q)\,, & n=2\,,\\
\frac 12 - \frac 1q\,, & n \geq 3 \,.
\end{cases}
\end{equation*}
Note that~\eqref{pqcondn} is equivalent to $\sigma(q)\ge \frac 1p\,.$
As in~\cite{SmSoBdry}, the crucial matter is to now show that if
$u_j$ satisfies the equation
$$
(D_t + P_j)u_j = F_j + G_j\,,
$$
then the following estimates are valid for $2 \leq j < N_\l$
\begin{align}
\|u_j\|_{L^p L^q(S_{j,k})} \lesssim \l^{s+ \frac 1p}
\theta_j^{\sigma(q)}\Big(&\|u_j\|_{L^\infty L^2(S_{j,k})} +
\|F_j\|_{L^1 L^2 (S_{j,k})}\notag\\
&+ \l^{\frac 14} \theta_j^{\frac 14} \|\langle \l^{\frac 12}
\theta_j^{-\frac 12}x_n\rangle^{-1}u_j\|_{ L^2 (S_{j,k})}\label{angest}\\
 &+ \l^{-\frac 14} \theta_j^{-\frac 14} \|\langle \l^{\frac
12} \theta_j^{-\frac 12} x_n\rangle^{2}G_j\|_{ L^2 (S_{j,k})}
\Big)\,.\notag
\end{align}
In case $j=N_\l$, that is $\theta_j = \l^{-\frac 13}$, then
\begin{equation}\label{angesttgt}
\|u_j\|_{L^p L^q(S_{j,k})}\lesssim \l^{s+ \frac 1p}
\theta_j^{\sigma(q)}\Big(\|u_j\|_{L^\infty L^2(S_{j,k})} +
\|F_j+G_j\|_{L^1 L^2 (S_{j,k})}\Big)\,.
\end{equation}

Let $c_{j,k}$ denote the term in parentheses on the right
of~\eqref{angest} (respectively the term in parentheses in~\eqref{angesttgt}
when $j=N_\l$).  
A modification of the arguments in~\cite[\S 6]{SmSoBdry} show
that if $k(j)$ denotes any sequence of values of $k$ for which the
slabs $S_{j,k(j)}$ are nested, that is, $S_{j+1,k(j+1)} \subset
S_{j,k(j)}$, then
\begin{equation}\label{fluxest}
\sum_{j}c_{j,k(j)}^2 \lesssim \|u_\l\|_{L^\infty L^2(S)}^2 +
\|(i\partial_t+P_\l)u_\l\|_{L^1 L^2(S)}^2\,.
\end{equation}
The key modification arises from the fact that the symbol $p_\l(x,\xi)$
is not homogeneous of degree 1 in $\xi$. This changes the form of the
conjugation
of $Q_\mu$ (which is the operator $P_\l$ after a space-time rescaling
by $\theta_j$), by the wave packet transform $T_\mu$, which occurs
on the bottom of \cite[p. 137]{SmSoBdry} and top of \cite[p. 145]{SmSoBdry}.
The new relation is
$$
T_\mu Q_\mu T_\mu^*=D_q+i\alpha+K\,.
$$
Here, $q$ is the symbol $p_j$ rescaled, and $D_q$ the Hamiltonian
field of $q$. The real valued function $\alpha$ is defined in \eqref{alphadef}
below. The operator $D_q+i\alpha$ is simply the conjugation of $D_q$
by the unimodular function $\exp(i\psi(t,x,\xi))$. This conjugation
does not affect the arguments of \cite[\S 6]{SmSoBdry}, since the
only estimates used on $K$ are absolute value bounds \cite[(6.21)]{SmSoBdry},
that follow from the estimates \cite[(6.31)]{SmSoBdry}.
We also note that the estimates in~\cite{SmSoBdry} use 
$(i\partial_t+P_\l)u_\l \in L^2 L^2(S)$, 
but as noted after Theorem \ref{dyadicthm} above this is unimportant.

The estimates of Theorem \ref{dyadicthm}
will then follow from \eqref{angest}-\eqref{angesttgt} and
the branching argument on~\cite[p. 118]{SmSoBdry}.

In the proof of \eqref{angest}-\eqref{angesttgt}, we will
from now on work with a fixed $j$, and will abbreviate $\theta_j=\theta$.
We work with an angle-dependent rescaled $u_j$,
setting
$$
u(t,x) = u_j(\theta_j t, \theta_j x)\,, 
\quad F(t,x) = \th_j F_j(\theta_j t,\theta_j x)\,,\quad
G(t,x) = \th_j G_j(\theta_j t, \theta_j x)\,,
$$
and $q(x,\xi) = \theta_j P_j(\theta_j x, \theta_j^{-1} \xi)$. Set
$\mu = \l \th_j$, so that $q(x,\xi) \approx \mu$ when $|\xi| \approx
\mu$. Additionally, if $|\xi| \approx \mu$, then
$q(x,\xi)$ satisfies the following estimates; see \cite[(4.1)]{SmSoBdry}.
\begin{equation}\label{qest}
|\prtl_x^\beta \prtl_\xi^\alpha q(x,\xi)| \lesssim
\begin{cases}
\mu^{1-|\alpha|}\,, & \text{ if }\, |\beta|=0\,,\\
c_0\,\bigl(\,1+\mu^{(|\beta|-1)/2} \theta_j \langle \mu^{\frac 12} x_n
\rangle^{-N}\bigr)\mu^{1-|\alpha|}\,, & \text{ if }\, |\beta|\geq 1\,.
\end{cases}
\end{equation}

We then have
$$
D_t u - q(x,D) u = F + G\,,
$$
and the frequency localization condition holds
\begin{equation*}
\supp(\widehat{u}(t,\cdot)) \subset 
\begin{cases}
|\xi'| \in [\tfrac 14\mu, 4\mu]\,,\, |\xi_n| \in
[\tfrac 14\mu \th, 2 \mu \th ]\,, & 
\theta  > \mu^{-\frac 12}\,, \\
|\xi'| \in [\tfrac 14\mu, 4\mu]\,,\, |\xi_n| \leq \mu^{\frac 12}\,, & 
\theta  = \mu^{-\frac 12}\,.
\end{cases}
\end{equation*}
After translation in time, the estimates~\eqref{angest} reduce to
showing that, over the slab $S=[0,\veps] \times \RR^n\,,$
\begin{align}
\|u\|_{L^p L^q(S)} \lesssim \mu^{s+ \frac 1p}
\theta^{\sigma(q)}\Big(&\|u\|_{L^\infty L^2(S)} +
\|F\|_{L^1 L^2 (S)}\notag\\
&+ \mu^{\frac 14} \theta^{\frac 12} \|\langle \mu^{\frac 12}x_n\rangle^{-1}u\|_{ L^2 (S)}\label{angrescale}\\
 &+ \mu^{-\frac 14} \theta^{-\frac 12} \|\langle \mu^{\frac 12}x_n\rangle^{2}G\|_{ L^2 (S)}
\Big)\,.\notag
\end{align}
The estimates in~\eqref{angesttgt} reduce to showing that, for
$\theta = \mu^{-\frac 12}\,,$
\begin{equation}\label{tgtrescale}
\|u\|_{L^p L^q(S)}\lesssim \mu^{s+ \frac 1p}
\theta^{\sigma(q)}\Big(\|u\|_{L^\infty L^2(S)} +
\|F+G\|_{L^1 L^2 (S)}\Big)\,.
\end{equation}

To establish the inequalities~\eqref{angrescale} and~\eqref{tgtrescale},
we use a wave packet transform to construct a suitable representation of $u$.  
Define the linear operator $T_\mu$ on Schwartz class functions by
$$
(T_\mu f)(x,\xi) = \mu^{\frac n4} \int e^{-i\langle \xi, y-x \rangle}
g(\mu^{\frac 12} (y-x)) f(y)\;dy\,,
$$
where we fix $g$ a radial Schwartz class function, with $\widehat{g} $ 
supported in a ball of small radius $c$.
Taking $\|g\|_{L^2(\RR^n)} = (2\pi)^{-\frac n2}\,,$ 
it holds that $T^*_\mu T_\mu =I$
and $\|T_\mu f\|_{L^2(\RR^{2n}_{x,\xi})}= \|f\|_{L^2(\RR^n_y)}$.
We set
$$
\tilde{u}(t,x,\xi) = (T_\mu u(t,\cdot))(x,\xi)\,.
$$
By Lemma 4.4 of~\cite{SmSoBdry}, we may write
\begin{multline}\label{uflow}
\bigl(\prtl_t - d_\xi q(x,\xi) \cdot d_x + d_x q(x,\xi) \cdot d_\xi
+iq(x,\xi)-i\xi\cdot d_\xi q(x,\xi)\bigr)\tilde{u}(t,x,\xi)\\
 = \tilde{F}(t,x,\xi) + \tilde{G}(t,x,\xi)\,,
\end{multline}
where, over $\tilde{S} = [0,\veps] \times \RR^{2n}_{x,\xi}$, the quantity
$$
\|\tilde{F}\|_{L^1 L^2(\tilde{S})} + \mu^{-\frac 14} \theta^{-\frac 12}
\|\langle \mu^\hf x_n \rangle^2\tilde{G}\|_{L^2(\tilde{S})}
$$
is bounded by the right hand side of~\eqref{angrescale}
when $\theta > \mu^{-\frac 12}\,,$ and the quantity
$$
\|\tilde{F} + \tilde{G}\|_{L^1 L^2(\tilde{S})}
$$
is bounded by the right hand side of~\eqref{tgtrescale} when
$\theta = \mu^{-\frac 12}$.  The proof of this lemma relies only on the
bounds~\eqref{qest}, and thus applies in our situation.  Also, given the 
compact support of $\widehat{g}$, it can be seen that the $\xi$ support of 
$\tilde{u}$, $\tilde{F}$, $\tilde{G}$ is contained in a set where 
$|\xi'| \approx \mu$  and $\xi_n \approx \theta \mu$
(or $|\xi_n| \lesssim \mu^{\frac 12}$ when $\theta = \mu^{-\frac 12}$).

Let $\Theta_{r,t}(x,\xi)$ denote the canonical transformation on 
$\RR^{2n}_{x,\xi}=
T^*(\RR^n_x)$ generated by the Hamiltonian flow of $q(x,\xi)$.  That is,
$\Theta_{r,t}(x,\xi)$  is the time $r$ solution of
\begin{equation}\label{qham}
\dot x = d_\xi q(x,\xi')\,, \qquad \dot {\xi} = -d_x q(x,\xi)\,,
\end{equation}
with initial conditions $(x(t),\xi(t))=(x,\xi)\,.$ Since $q(x,\xi)$ is 
independent of time, $\Theta_{r,t} = \Theta_{r-t,0}\,.$  Also define
\begin{equation}\label{alphadef}
\alpha(x,\xi)=q(x,\xi)-\xi\cdot d_\xi q(x,\xi)\,,\qquad 
\psi(t,x,\xi) = \int_0^t \alpha(\Theta_{s,t}(x,\xi))\;ds\,.
\end{equation}
It follows by time independence of $q$ that
$\int_r^t\alpha(\Theta_{s,t}(x,\xi))\;ds=\psi(t-r,x,\xi)\,.$

Equation \eqref{uflow} above allows us to write
\begin{multline*}
\tilde{u}(t,x,\xi)= e^{-i\psi(t,x,\xi)} \tilde{u}(0, \Theta_{0,t} (x,\xi))\\ 
+ \int_0^t e^{-i\psi(t-r,x,\xi)} \left(\tilde{F}(r, \Theta_{r,t}
(x,\xi) )+  \tilde{G}(r, \Theta_{r,t}
(x,\xi) ) \right)\;dr\,.
\end{multline*}

In the next section we will establish the following estimates
for solutions to the homogeneous flow equation,
\begin{theorem}\label{coniclocthm}
Suppose $f \in L^2(\RR^{2n}_{x,\xi})$ is supported in a set of the form
\begin{equation}\label{conicloc}
\begin{cases}
|\xi'| \approx \mu\,,\, \;|\xi_n| \approx \mu \theta\,, & 
\theta  > \mu^{-\frac 12}\,, \\
|\xi'| \approx \mu\,,\, \;|\xi_n| \leq \mu^{\frac 12}\,, & 
\theta  = \mu^{-\frac 12}\,.
\end{cases}
\end{equation}
Define $W\! f (t,x)= T_\mu^* 
\bigl[e^{-i\psi(t,\cdot)}(f \circ \Theta_{0,t})\bigr](x)\,.$ 
Then the following estimate holds for $s\ge 0,$ $p>2,$ and $q<\infty$
satisfying \eqref{pqcondn} and \eqref{scalingcond},
\begin{equation}\label{West}
\|W\! f\|_{L^p L^q (S)} \lesssim \mu^{s+\frac 1p} \theta^{\sigma(q)}
\|f\|_{L^2(\RR^{2n})}\,.
\end{equation}
For $f\in L^2(\RR^{2n}_{x,\xi})$ supported where $|\xi'|\le\mu$, $|\xi_n|\approx\mu$, estimate
\eqref{West} holds with $\theta=1$, for $s\ge 0,$ $p>2,$ and $q<\infty$
satisfying \eqref{scalingcond}.
\end{theorem}

\smallskip

Since $T_\mu^* T_\mu =I\,,$ it follows by the preceeding steps and
variation of parameters that this implies the
estimates~\eqref{tgtrescale}, as well as the estimates~\eqref{angrescale}
in case $\tilde{G} \equiv 0$.
The reduction of the estimates~\eqref{angrescale} to Theorem \ref{coniclocthm}
for $\tilde{G} \neq 0$
requires the $V^2_q$ spaces introduced by Koch and Tataru~\cite{KT}, and
follows exactly the arguments on \cite[p.~124--126]{SmSoBdry}.
The key fact used in that proof about the Hamiltonian flow of $q$ is that
$\dot x_n\approx\theta$ on the support of $\tilde u(t,x,\xi)$,
which holds in our case.

%%%%%%%%%%%%%%%%%%%%%%%%%%%%%%%%%%%%%%%%%%%%%%%%%%%

\newsection{Homogeneous estimates}

\noindent
In this section we prove Theorem \ref{coniclocthm}.  By duality,
it suffices to show that
\begin{equation}\label{dualest}
\|W W^* F \|_{L^p L^q(S)} \lesssim \mu^{2(s + \frac 1p)} \th^{2\sigma(q)} 
\|F\|_{L^{p'} L^{q'} (S)}\,,
\end{equation}
where $\widehat F(t,\xi)$ is supported as in \eqref{conicloc},
and we recall that
$$
s=n\bigl(\tfrac 12-\tfrac 1q\bigr)-\tfrac 2p\,,\qquad
\sigma(q)=
\begin{cases} 
\tfrac 23\bigl(\tfrac 12-\tfrac 1q\bigr)\,,\; &n=2\,,\\
\tfrac 12-\tfrac 1q\,,&n\ge 3\,.
\end{cases}
$$

Let $W_t$ denote the fixed time operator $W_t f = W\!f(r,x)|_{r=t}$.  
We will show that
\begin{equation}\label{dispersive}
\|W_r W_t^* \|_{L^1 \to L^\infty} \lesssim  \mu^{\frac n2} 
(\mu^{-1} + |t-r|)^{-\frac{n-1}{2}}(\mu^{-1} \th^{-2} + |t-r|)^{-\frac 12}\,,
\end{equation}
\begin{equation}\label{energy}
\|W_r W_t^* \|_{L^2 \to L^2} \lesssim 1\,.
\end{equation}
Interpolation of these estimates yields
\begin{equation}\label{qq'est}
\|W_r W_t^* \|_{L^{q'} \to L^q} \lesssim  \mu^{\frac n2 (1-\frac 2q)} 
(\mu^{-1} + |t-r|)^{-\frac{n-1}{2}(1-\frac 2q)}
(\mu^{-1} \th^{-2} + |t-r|)^{-\frac 12(1-\frac 2q)}.
\end{equation}
For $n \geq 3$, we have $\frac 2p \le 1-\frac 2q\le \frac{n-1}2
\bigl(1-\frac 2q\bigr)$, hence
we may ignore the term $|t-r|$ in the last factor to obtain
\begin{equation}\label{decayest}
\|W_r W_t^* \|_{L^{q'} \to L^q} \lesssim 
\mu^{2(s+\frac 1p)}\th^{2\sigma(q)} |t-r|^{-\frac 2p}\,.
\end{equation}
In case $n=2$, 
we use that $\theta\le 1$ to bound
$$
\|W_r W_t^* \|_{L^{q'} \to L^q} \lesssim  
\mu^{\frac 43(1-\frac 2q)} \th^{\frac 23(1-\frac 2q)}
\bigl(\mu^{-1}+|t-r|\bigr)^{-\frac 23(1-\frac 2q)}\,.
$$
Since $\frac 2p\le\frac 23(1-\frac 2q)$, in this case we again have
\eqref{decayest}.
In either case, the Hardy-Littlewood-Sobolev theorem gives~\eqref{dualest}.

In case $\theta=1$, as for the normal piece, estimate \eqref{decayest} follows
for all $p,q$ satisfying \eqref{scalingcond} with $s\ge 0$. Hence 
for the normal piece the condition \eqref{pqcondn} is not necessary.

The inequality \eqref{energy} follows from the fact that $T_\mu$ is an
isometry and $\Theta_{0,t}(x,\xi)$ is a symplectomorphism, and hence
preserves the measure $dx\,d\xi$.  The remainder
of this section is devoted to proving~\eqref{dispersive}.

The action of $W_r W_t^*$ on a function $h(y)$ can be expressed as 
integration against an integral kernel $K(r,x;t,y)$, defined by the formula
$$
\mu^{\frac{n}{2}} \int e^{i\langle \zeta, x-z \rangle - i
\psi(r-t,x,\zeta)
-i \langle \zeta_{t,r},
y-z_{t,r} \rangle} g(\mu^\hf(y-z_{t,r}))g(\mu^\hf(x-z))\beta_\theta(\zeta)
\,dz\, d\zeta\,.
$$
Recall that $\widehat{g}$ is supported in a ball of small radius and 
$f(x,\xi)$ is assumed to have $\xi$ support in a set of the 
form~\eqref{conicloc}, which is essentially preserved by the Hamiltonian
flow of $q$ for time $\veps$.  
Hence $\beta_\th(\zeta)$ can be taken to be a 
smooth cutoff to a set of the form~\eqref{conicloc}.
For convenience, we take $\beta_\theta(\zeta)$ to be a product of
a cutoff in $\zeta'$ and a cutoff in $\zeta_n$.

Since $\Theta_{t,r} = \Theta_{t-r,0}$,  it suffices to consider the case 
$r=0$. We abbreviate $(z_{t,0},\zeta_{t,0})$ by $(z_t,\zeta_t)$, so that
\begin{equation}\label{hamsystem}
\prtl_t z_t(z,\zeta) = d_\zeta q(z,\zeta)\,, 
\qquad \prtl_t \zeta_t(z,\zeta) = -d_z q(z,\zeta)\,, 
\qquad (z_0, \zeta_0) = (z, \zeta)\,.
\end{equation}
The kernel $K(0,x;t,y)$ takes the form
\begin{equation}\label{kformula}
\mu^{\frac n2} \int e^{i \langle \zeta, x-z \rangle + i\psi(t,z_t,\zeta_t) 
- i \langle \zeta_t, y-z_t \rangle} g(\mu^{\frac 12}(y-z_t)) 
g(\mu^{\frac 12}(x-z)) \beta_\theta(\zeta)\,dz\,d\zeta\,.
\end{equation}

\begin{theorem}
Suppose $(z_t(z,\zeta),\zeta_t(z,\zeta))$ are defined by~\eqref{hamsystem} 
and $d_z$, $d_\zeta$ denote the $z$ and $\zeta$ gradient operators.  Then if 
$|\zeta| \approx \mu$, and $\zeta_n \approx \mu \theta$,
or $|\zeta_n| \lesssim \mu^{\frac 12}$ in the case $\theta = \mu^{-\frac 12}$, 
the following bounds hold,
\begin{align}
|d_z z_t -I| \lesssim t\,, & & |d_\zeta z_t|\lesssim \mu^{-1}t\,, 
\label{1stderivs}\\
|d_\zeta \zeta_t -I | \lesssim t\,, & & |d_z \zeta_t | \lesssim \mu\,,
\notag
\end{align}
as well as the more precise estimate
\begin{equation}\label{curv}
\left| d_\zeta z_t - \int_0^t d^2_\zeta q(z_s,\zeta_s)\;ds\right| 
\lesssim \mu^{-1}t^2\,.
\end{equation}
Furthermore, for second order derivatives we have
\begin{align}
|d^2_z z_t|\lesssim \langle \mu^{\frac 12} t \rangle\,, & & |d^2_z \zeta_t| 
\lesssim \mu^{\frac 32}\,,\label{zderivs}\\
|d_z d_\zeta z_t | \lesssim \mu^{-1}t \langle \mu^{\frac 12} t \rangle\,, & & 
|d_z d_\zeta \zeta_t | \lesssim \langle \mu^{\frac 12} t \rangle\,.
\label{mixderivs}
\end{align}
Finally, for $l\geq 2$ we have
\begin{equation}\label{zetaderivs}
\mu^l|d^l_\zeta z_t | + \mu^{l-1} |d^l_\zeta \zeta_t | 
\lesssim t \langle \mu^{\frac 12} t \rangle^{l-1}\,.
\end{equation}
\end{theorem}

\begin{proof}
The proof is a rescaled version of Theorem 5.1 and Corollary 5.2 of
\cite{SmSoBdry}, but for completeness we sketch the details here.

Differentiating Hamilton's equations one obtains
$$
\prtl_t
\begin{bmatrix}
d z_t \\ d \zeta_t
\end{bmatrix}
=M(z_t, \zeta_t)
\begin{bmatrix}
d z_t \\ d \zeta_t
\end{bmatrix},
\qquad M(z,\zeta) =
\begin{bmatrix} \phantom{-}d_z d_\zeta q & \phantom{-}d_\zeta d_\zeta q\\
-d_z d_z q & -d_\zeta d_z q
\end{bmatrix}.
$$
To keep all terms of the same order in $\mu$,
we take the following rescaled equation,
\begin{equation}\label{diffham}
\prtl_t
\begin{bmatrix}
\phantom{\mu^{-1}}d_z z_t & \mu \,d_\zeta z_t \\ \mu^{-1} d_z \zeta_t & 
\phantom{\mu}d_\zeta \zeta_t
\end{bmatrix}
=M_\mu(z_t, \zeta_t)
\begin{bmatrix}
\phantom{\mu^{-1}}d_z z_t & \mu \,d_\zeta z_t \\ \mu^{-1} d_z \zeta_t & 
\phantom{\mu}d_\zeta \zeta_t
\end{bmatrix},
\end{equation}
where
$$
M_\mu (z,\zeta) =
\begin{bmatrix} \phantom{-\mu^{-1}}d_z d_\zeta q & \phantom{-} \mu 
\,d_\zeta d_\zeta q\\
-\mu^{-1} d_z d_z q & \phantom{\mu\,}-d_\zeta d_z q
\end{bmatrix}.
$$
The key estimate on $M_\mu$ is that, for $j+k =2\,,$
\begin{equation}\label{keyMest}
\int_0^t |(d^{j}_z d^k_\zeta q) (z_s, \zeta_s)| \,ds \lesssim
\begin{cases}
\mu^{-1}t\,, & \text{if } k =2\,,\\
t\,, & \text{if } j =k= 1\,,\\
\mu\,, & \text{if } j =2\,.
\end{cases}
\end{equation}
This follows from~\eqref{qest} and the property 
$|(\prtl_t z_t)_n | \approx \theta$ for $t \in [0,\veps]$, when 
$\theta > \mu^{-\frac 12}$.  When $\theta = \mu^{-\frac 12}$,
the estimates \eqref{qest} are uniform over $|\beta|\le 2$, and 
\eqref{keyMest} also follows. Gronwall's lemma now gives that
$$
|d_z z_t| +\mu\, |d_\zeta z_t| + \mu^{-1} |d_z \zeta_t | + |d_\zeta \zeta_t| 
\lesssim 1\,.
$$
Integrating~\eqref{diffham} and using \eqref{keyMest}
yields~\eqref{1stderivs}.  The estimate $|d_\zeta \zeta_t -I | \lesssim t$ can 
then be substituted in the integral equation for $\prtl_\zeta z_t$ to 
give~\eqref{curv}.

To show the higher order estimates~\eqref{zetaderivs}, 
we work with the equation
\begin{equation*}
\prtl_t
\begin{bmatrix}
\mu^l d_\zeta^l z_t \\ \mu^{l-1} d_\zeta^{l} \zeta_t
\end{bmatrix}
=M_\mu(z_t, \zeta_t)
\begin{bmatrix}
\mu^l d_\zeta^l z_t \\ \mu^{l-1} d_\zeta^{l} \zeta_t
\end{bmatrix}
+
\begin{bmatrix}
E_1(t) \\ E_2(t)
\end{bmatrix}.
\end{equation*}
Here $E_1(t)$ is a sum of terms of the form
$$
(\mu^k d^j_z d^{k+1}_\zeta q)(z_t, \zeta_t)
(\mu^{l_1}d_\zeta^{\,l_1} z_t) \dots
 (\mu^{l_j}d_\zeta^{\,l_j} z_t)(\mu^{l_{j+1}-1}d_\zeta^{\,l_{j+1}} \zeta_t)
\dots (\mu^{l_{j+k}-1}d_\zeta^{\,l_{j+k}} \zeta_t).
$$
Similarly, $E_2(t)$ can be written as a sum of such terms, but with
the first factor
replaced by $(\mu^{k-1} d^{j+1}_z d^{k}_\zeta q)(z_t, \zeta_t)$. In either case, 
$l_m < l$ for all $m$ and $l_1 + \cdots l_{j+k}=l$.  
The estimate~\eqref{zetaderivs} now follows by an inductive argument 
which uses the bounds
$$
\int_0^t \mu^{k-1}|(d_z^{j+1} d_\zeta^k q)(z_s, \zeta_s)|\;ds \lesssim
\begin{cases}
t\,, & \text{if } j=0\,,\\
\mu^{\frac{j-1}2}\,, & \text{if } j\geq 1\,.
\end{cases}
$$
Estimates~\eqref{zderivs} and~\eqref{mixderivs} follow similarly; 
see the proof of Theorem 5.1 in \cite{SmSoBdry}.
\end{proof}

We start the proof of \eqref{dispersive} by noting that
absolute bounds on the integrand in~\eqref{kformula} easily yield
$$
|K(0,x;t,y)| \lesssim \mu^n \theta\,,
$$
which gives \eqref{dispersive} for $0 \leq t \leq \mu^{-1}$.
We next consider the cases $\mu^{-1} \leq t \leq \mu^{-1}\theta^{-2}$ and  
$\mu^{-1}\theta^{-2} \leq t \leq \veps$ separately.  In these two cases, 
we will respectively
integrate by parts in \eqref{kformula} with the two vector fields
$$
L'=\frac{1 - i\mu t^{-1} (x-z -d_\zeta \zeta_t \cdot 
(y-z_t) )'\cdot d_{\zeta'} }{1 +  \mu t^{-1} 
\bigl|(x-z -d_\zeta \zeta_t \cdot (y-z_t) )'\bigr|^2}
$$
$$
L=\frac{1 - i\mu t^{-1} (x-z -d_\zeta \zeta_t \cdot(y-z_t) )
\cdot d_\zeta}{1+\mu t^{-1}\bigl|x-z -d_\zeta \zeta_t \cdot (y-z_t)\bigr|^2}
$$

Both $L'$ and $L$ preserve the phase function in~\eqref{kformula}.  
This can be seen by noting that 
$\psi(t,z_t,\zeta_t)=\int_0^t\alpha(s,z_s,\zeta_s)\,ds$, and observing that
$$
\prtl_{\zeta_i}\left(\int_0^{t} q(z_s,\zeta_s) -
\zeta_r\cdot(d_\zeta q)(z_s,\zeta_s)\,ds \right) + \zeta_t
\cdot \prtl_{\zeta_i} z_t=0\,.
$$
The expression vanishes at $t=0$ since $d_\zeta z_0=0$, and Hamilton's
equations
show that the derivative of the expression with respect to $t$ vanishes
identically.

We begin with the case where $ \mu^{-1} \leq t \leq \mu^{-1} \theta^{-2}\,.$ 
Recall that $\beta_\theta(\zeta)$ is the product of smooth cutoffs to
$|\zeta_n|\approx \theta\mu$ and $|\zeta'|\approx\mu$.
Let $\{\xi_m'\}$ be a collection of $\approx(\mu t)^{\frac {n-1}2}$ vectors
on the lattice of spacing $\mu^{\frac 12}t^{-\frac 12}$, and $\phi$
a cutoff so that
$$
\beta_\theta(\zeta) = \sum_m \beta_\th (\zeta) \phi_m(\zeta')\,,
$$
where $\phi_m(\zeta')=\phi(\mu^{-\frac 12}t^{\frac 12}(\zeta'-\xi'_m))\,.$

Define $K_m(t,x,y)$ as the integral in~\eqref{kformula} with 
$\beta_\th(\zeta)$ replaced by $\beta_\th (\zeta) \phi_m(\zeta')$ so that
$K(0,x;t,y) = \sum_m K_m(t,x,y)\,.$

By the estimates~\eqref{1stderivs} 
and \eqref{zetaderivs}, we have that for $k \geq 1$,
$$
|(\mu^{\frac 12} t^{-\frac 12} d_\zeta)^k a| \lesssim 1, \quad
\text{for}\quad a(t,z,\zeta) = 
\mu^{\frac 12} t^{-\frac 12} z_t \;\text{  or  }\; a(t,z,\zeta) =  
t^{-\frac 12} d_\zeta \zeta_t,
$$
which holds not just for $t \in [\mu^{-1}, \mu^{-1}\th^{-2}]$ but for all
$t \in [\mu^{-1}, \veps]$.  Furthermore, 
$|(\mu^{\frac 12} t^{-\frac 12} d_{\zeta'} )^k\phi_m(\zeta')\beta_\th(\zeta)|
\lesssim 1 $, since we do not differentiate in $\zeta_n$.  
Therefore, integration by parts yields the following 
upper bound on $K_m(t,x,y)$,
\begin{multline*}
\mu^{\frac{n}{2}} \int_{\RR^n \times \supp(\beta_\theta\phi_m)}
\bigl(1+\mu^{\hf} t^{-\hf} |(x-z -d_\zeta \zeta_t \cdot (y-z_t))'|
\bigr)^{-N}\\
\times\bigl(1+\mu^\hf|y-z_t|\bigr)^{-N}\bigl(1+\mu^{\hf}|x-z|\bigr)^{-N}
\,dz\, d\zeta\,.
\end{multline*}
We set $\xi_m = (\xi_m', \mu\th)$.  Since $t \leq \mu^{-1} \theta^{-2}\,,$ 
we have for $\zeta \in \supp(\beta_\theta\phi_m)$,
\begin{equation}\label{zetamineq}
|\zeta -\xi_m| \lesssim \mu^{\frac 12} t^{-\frac 12}\,.
\end{equation}
Recall that $z_t=z_t(z,\zeta)$ is the spatial component of
$\Theta_{t,0}(z,\zeta)$. We let $x^m_t=z_t(x,\xi_m)$
denote the spatial component of $\Theta_{t,0}(x,\xi_m)$.
We then claim that, for $\zeta\in \supp(\beta_\theta\phi_m)$,
\begin{equation}\label{xzapprox}
\mu^{\hf}t^{-\hf}
|\,x-z -d_\zeta \zeta_t \cdot (x^m_t-z_t) | \lesssim 1 + \mu\,|x-z|^2 .
\end{equation}
Assuming this for the moment, we dominate the integrand for $K_m$ by
\begin{equation}\label{integrand}
\bigl(1+\mu^{\hf}t^{-\hf}|( d_\zeta \zeta_t \cdot
(y-x^m_t))'|\bigr)^{-N}\bigl(1+\mu^\hf|y-z_t|\bigr)^{-N}
\bigl(1+\mu^{\hf}|x-z|\bigr)^{-N}\,.
\end{equation}
By~\eqref{1stderivs} and \eqref{zetamineq}, 
we have $|x^m_t -z_t| \lesssim \mu^{-\frac 12} t^{\frac 12}+|x-z|$.  
Thus, since $|d_\zeta \zeta_t -I|\lesssim |t|$, we conclude that
\begin{align*}
\mu^\hf t^{-\hf}|( d_\zeta \zeta_t \cdot (y-x^m_t))'| &
\gtrsim \mu^\hf t^{-\hf}|(y-x^m_t)'| - \mu^\hf t^{\hf}|y-x^m_t|\\
 &\gtrsim \mu^\hf t^{-\hf}|(y-x^m_t)'| - 
\mu^\hf t^{\hf}\bigl(|y-z_t|+|x-z|\bigr) - |t|\,.
\end{align*}
The negative terms on the right here are small compared to the
last two terms in~\eqref{integrand}.  
Therefore, we have
$$
|K_m(t,x,y)|\lesssim \mu^{\frac{n+1}2}\theta\, t^{-\frac{n-1}2} 
\bigl(1+\mu^{\frac 12} t^{-\frac 12}|(y-x^m_t)'|\bigr)^{-N},
$$
which follows by observing the rapid decay of the integrand in $z$, and that 
the volume of $\supp(\phi_m\beta_\th)$ is comparable to
$\mu^{\frac{n+1}2}\theta\, t^{-\frac{n-1}2}$.

We next observe that, by~\eqref{curv} and the estimate
$$
\bigl\|d^2_\zeta q(z,\zeta) -2 \mu^{-1} I\bigr\| \lesssim 
\mu^{-1}\|g^{ij}-\delta_{ij}\| 
\lesssim c_0\mu^{-1}\,,
$$
we have that
$$
|(x^m_t - x^l_t) - 2\mu^{-1}t(\xi_m -\xi_l)| 
\ll\mu^{-1}t\, |\xi_m -\xi_l|\,,
$$
and since $|\xi_m-\xi_l|=|\xi_m'-\xi_l'|$, we conclude that
\begin{equation}\label{xtsep}
\mu^{\frac 12} t^{-\hf}|( x^m_t-x^l_t)'| 
\approx  \mu^{-\frac 12} t^{\hf}|\xi_m' -\xi_l'|\,.
\end{equation}
Since the $\xi_m'$ lie on a $\mu^{\frac 12}t^{-\frac 12}$ spaced lattice,
we may sum over $m$ to obtain
$$
|K(0,x;t,y)|\lesssim \mu^{\frac{n+1}2}\theta\, t^{-\frac{n-1}2} 
\sum_m 
\bigl(1+\mu^{\frac 12} t^{-\frac 12}|(y-x^m_t)'|\bigr)^{-N} 
\lesssim \mu^{\frac{n+1}2}\theta\, t^{-\frac{n-1}2}\,,
$$
yielding \eqref{dispersive} for $\mu^{-1}\le t\le \mu^{-1}\theta^{-2}$.

To handle the case $\mu^{-1} \theta^{-2} \leq t \leq \veps\,,$ we modify 
the above proof by considering an $\mathcal{O}(\mu^{\frac{n}2 } t^{\frac n2})$ 
collection of vectors $\{\xi_m\}$ in a
$\mu^{\frac 12}t^{-\frac 12}$ spaced lattice in $\R^n$, 
and an associated partition 
$\phi_m (\zeta) = \phi(\mu^{-\frac 12} t^{\frac 12}(\zeta - \xi_m))$, 
satisfying
$$
\beta_\theta(\zeta) = \sum_m \beta_\th (\zeta) \phi_m(\zeta)\,.
$$
We now define $K_m(t,x,y)$ as the integral in~\eqref{kformula} with 
$\beta_\th(\zeta)$ replaced by $\beta_\th (\zeta) \phi_m(\zeta)$.  
Here, since $\mu^{\hf} t^{-\hf}\le\mu\theta$, we have
$$
|(\mu^{\frac 12} t^{-\frac 12} d_{\zeta} )^k\phi_m(\zeta)\beta_\th(\zeta)|
\lesssim 1\,.
$$
Integrating by parts with respect to the vector field $L$ now shows that 
$K_m(t,x,y)$ is bounded by
\begin{multline*}
\mu^{\frac{n}{2}} \int_{\RR^n \times \supp(\beta_\theta\phi_m)}
\bigl(1+\mu^{\hf} t^{-\hf} |x-z -d_\zeta \zeta_t \cdot (y-z_t) |\bigr)^{-N}\\
\times\bigl(1+\mu^\hf|y-z_t|\bigr)^{-N}\bigl(1+\mu^{\hf}|x-z|\bigr)^{-N}
\,dz\, d\zeta\,.
\end{multline*}
Using~\eqref{xzapprox}, which holds for any $t \in [\mu^{-1}, \veps]\,,$ 
we proceed as before and conclude that
$$
|K_m(t,x,y)|\lesssim \mu^{\frac{n}2}t^{-\frac{n}2} 
\bigl(1+\mu^{\frac 12} t^{-\frac 12}|y-x^m_t|\bigr)^{-N}.
$$
The $n$-dimensional analogue of~\eqref{xtsep} is valid here, 
so we may use the 
spacing of the $\xi_m$ as above to sum over $m$ and obtain
\eqref{dispersive} for $t>\mu^{-1}\theta^{-2}$, that is,
$$
|K(0,x;t,y)|\lesssim \mu^{\frac{n}2} t^{-\frac{n}2}\,.
$$

Returning to~\eqref{xzapprox}, we first observe that by estimating the 
Taylor remainder using~\eqref{zderivs}--\eqref{zetaderivs}
and \eqref{zetamineq}, the following holds
\begin{equation*}
\mu^{\hf}t^{-\hf}|x^m_t - z_t-(d_z z_t)(x-z)- (d_\zeta z_t)(\xi_m-\zeta)|
\lesssim 1 +\mu |x-z|^2\,.
\end{equation*}
Furthermore by~\eqref{1stderivs} and \eqref{zetamineq}, we have
$$
\mu^{\hf}t^{-\hf}|(d_\zeta z_t) (\xi_m -\zeta)| \lesssim 1\,.
$$
From the fact that $(z,\zeta)\rightarrow (z_t,\zeta_t)$ 
is a symplectic transformation, we have
$$
\prtl_{\zeta_i}\zeta_t \cdot \prtl_{z_j} z_t - \prtl_{\zeta_i }
z_t \cdot \prtl_{z_j} \zeta_t=\delta_{ij}\,,
$$
where $\;\cdot\;$ pairs the $z_t$ and $\zeta_t$ indices.  Lastly,
by \eqref{1stderivs},
$$
\mu^{\hf} t^{-\hf}|d_\zeta z_t | \, |d_z \zeta_t| \, |x-z|
\lesssim \mu^{\frac 12} t^{\frac{1}{2}}|x-z| \le \mu^{\hf}|x-z|\,.
$$
These facts now combine to yield~\eqref{xzapprox}.\qed

%%%%%%%%%%%%%%%%%%%%%%%%%%%%%%%%%%%%%%%%%%%%%%%%%%%

\newsection{Applications to semilinear Schr\"{o}dinger equations on exterior domains}\label{wpextdom}

\noindent
In this section,
we assume that $\Omega = \RR^3\setminus \mathcal{K}$ is the domain exterior 
to a smooth non-trapping obstacle $\mathcal{K}$ (or any exterior domain 
where~\eqref{locsmooth} holds).  We consider the initial value problem for 
the following family of semilinear Schr\"{o}dinger equations 
in $3+1$ dimensions,
\begin{equation}\label{critnls}
i\prtl_t u + \Delta u \pm |u|^{r-1} u =0\,, \qquad u(0,x) =f(x)\,,
\end{equation}
satisfying homogeneous Dirichlet or Neumann boundary conditions
\begin{equation}\label{nlsbc}
u(t,x)\big|_{\prtl \Omega} =0\,, \qquad \text{ or } \qquad 
\prtl_\nu u(t,x)\big|_{\prtl \Omega} =0\,.
\end{equation}
Precisely, by a solution to \eqref{critnls}-\eqref{nlsbc}, 
we understand that, with $F(u)=\pm |u|^{r-1}u$, and
$u(t)$ denoting the function $u(t,\cdot)$,
\begin{equation}\label{lineqn}
u(t)=e^{it\Delta}f+i\int_0^t e^{i(t-s)\Delta}\,F(u(s))\,ds\,,
\end{equation}
where $\exp(it\Delta)$ is the unitary Schr\"odinger propagator defined
using the Dirichlet or Neumann spectral resolution. Defocusing means that
$F(u)=-|u|^{r-1}u$. Since we will work with $H^1$ data, the boundary conditions
required of the initial data in the Dirichlet case are that $f$ vanish on 
$\partial\Omega$;
in the Neumann case the boundary conditions are void, i.e.~$f$ is the 
restriction to $\Omega$ of a general function in $H^1(\R^3)$.

Planchon and Vega showed in~\cite{planvega} that, for $1<r<5$ and
defocusing nonlinearities, one has global existence of solutions to the 
Dirichlet problem for $f \in H^1$.  A crucial ingredient in their 
proof was the estimate in Theorem~\ref{extthm} with $p=q=4$ and $s = \frac 14$.
They combined this with local smoothing estimates near the boundary
to obtain well-posedness of
solutions for time $T>0$ depending on $\|f\|_{H^1}$.  Conservation of 
energy and mass
\begin{align*}
\int_\Omega \frac 12 |d_x u(t)|^2  + \frac 1{r+1} |u(t)|^{r+1} \;dx & =
\text{constant in }t\\
\int_\Omega |u(t)|^2 \;dx & =
\text{constant in }t
\end{align*}
can then be used to establish global existence of solutions.  Given that 
Theorem~\ref{extthm} holds for either Dirichlet or Neumann boundary 
conditions, we remark that our estimates can be used in the same way 
to obtain global existence of the solution to the Neumann problem.

In the critical case $r=5$, local well-posedness for 
solutions with $f \in H^1$, and global well-posedness for small data,
was proved by Ivanovici in~\cite{ivanovici}, under
the assumption that $\mathcal{K}$ is strictly convex, and $u$ satisfies
Dirichlet conditions.
These assumptions are necessary for the Melrose-Taylor parametrix construction
used to establish Strichartz estimates in~\cite{ivanovici}.
The Strichartz estimates were in fact shown to hold
in this setting for the full range
of $p,q$ satisfying \eqref{scalingcond}, provided $p>2$.
Recently, Ivanovici and Planchon in \cite{IP} extended the well-posedness
results to the case of general non-trapping $\mathcal{K}$, 
for both Dirichlet and Neumann
conditions, using certain $L^5_xL^2_t$ estimates from \cite{SmSoBdry}.

Here, we use our estimates to give a simple proof of the well-posedness
results for $H^1$ data for the critical case $r=5$, and general non-trapping
$\mathcal K$. The proof proceeds by a contraction 
argument using a $L^4_tL^\infty_x$ Strichartz estimate. As in~\cite{IP},
the local results are valid for the Dirichlet and Neumann cases, and the
proof yields scattering for small data in the Dirichlet case.
Precisely, we establish the following.

\begin{theorem}\label{critnlsthm}
Suppose $f \in H^1(\Omega)$, where $f|_{\partial\Omega}=0$ if Dirichlet
conditions are assumed. Then there exists $T>0\,,$ 
and a solution $u(t,x)$ to equation \eqref{lineqn}
with $r=5$ on $[-T,T]\times\Omega$, satisfying
$$
u \in X\equiv C([-T,T];H^1(\Omega)) 
\cap L^4([-T,T];L^\infty(\Omega))\,,
$$
and the solution is
uniquely determined in this function space.  
Furthermore, if the data satisfies
$\|f\|_{H^1}\le \veps$ for $\veps$ sufficiently small, one can take $T=\infty$
in the case of Dirichlet conditions, and $T=1$
for Neumann conditions.
\end{theorem}

The key ingredient in Theorem~\ref{critnlsthm} will be the following Strichartz 
estimate for $u$ given by formula \eqref{lineqn}, and with $f$ and 
$F$ satisfying the given boundary condition,
\begin{equation}\label{l4stz}
\|u\|_{L^4_T L^\infty}  
\lesssim \|f\|_{H^1} + \|F\|_{L^1_T H^1}\,.
\end{equation}
Given \eqref{l4stz}, one applies differentiation and H\"older's inequality to obtain
\begin{equation}\label{l4holder}
\bigl\||u|^4 u\,\bigr\|_{L^1_T H^1} 
\lesssim \|u\|_{L^4_T L^\infty}^4 \|u\|_{L^\infty_T H^1}\,,
\end{equation}
noting that $|u|^4u$ satisfies Dirichlet conditions if $u$ does.

We then pose $u=u_0+v$, where $u_0(t)=\exp(it\Delta)f$, and take
$T$ small enough so that $\|u_0\|_{L^4_TL^\infty}<c$, for $c$ to
be determined.
Estimates \eqref{l4stz} and \eqref{l4holder}, together with
conservation of the $H^1$ norm under $\exp(it\Delta)$, show that,
for small $c$, the map
$$
v\rightarrow \int_0^te^{i(t-s)\Delta}\Bigl(|u_0(s)+v(s)|^4(u_0(s)+v(s))\Bigr)\,ds
$$
maps the ball $\|v\|_X\le c$ into itself.
Similar analysis shows that the map is in fact a contraction on 
this ball, for small $c$, yielding a fixed point $v$. 
If $\|f\|_{H^1}\le\veps$, then one can take $T=\infty$
for the Dirichlet case, or $T=1$ for the Neumann case.

For defocusing Neumann, energy and mass conservation
then yield global existence. For small norm Dirichlet
data, the proof implies $|u|^4u\in L^1(\R,H^1(\Omega))$. This yields that
such solutions scatter, in the sense that they asymptotically approach in
the $H^1$ norm a solution to the homogeneous equation.

In establishing \eqref{l4stz}, it suffices by the Duhamel principle to
consider $F=0$. The proof of~\eqref{l4stz} will be
obtained from the following cases of Theorem \ref{extthm},
$$
\|u\|_{L^{12}L^9}\lesssim \|f\|_{H^1}\,,\qquad\quad
\|u\|_{L^3L^9}\lesssim \|f\|_{H^{\frac 12}}\,.
$$
The second estimate could be expressed as controlling the $L^3W^{\frac 12,9}$
norm of $u$ in terms of $\|f\|_{H^1}$, and we would then apply
a fractional Gagliardo-Nirenberg inequality to control $\|u(t)\|_{L^\infty}$
by interpolating $L^9$ and $W^{\frac 12,9}$. We can avoid dealing with
fractional $L^p$ Sobolev spaces on exterior domains, 
however, by carrying out the same steps
more directly. The interpolation we will use is the following.

\begin{lemma}\label{gagnir}
Suppose that $\alpha_1,\alpha_2>0$, and $u=\sum_{j=0}^\infty u_j$, where
$$
\|u_j\|_{L^\infty}\le 
\min\bigl(\,2^{-j\alpha_1}\rho_1\,,\,2^{j\alpha_2}\rho_2\bigr)\,.
$$
Then 
$$
\|u\|_{L^\infty}\le C_{\alpha_1,\alpha_2}
\rho_1^{\frac {\alpha_2}{\alpha_1+\alpha_2}}
\rho_2^{\frac {\alpha_1}{\alpha_1+\alpha_2}}\,.
$$
\end{lemma}
\begin{proof}
The proof follows by summing the smaller of the bounds, i.e. separating
the sum depending on whether 
$2^j\ge(\rho_1/\rho_2)^{\frac{1}{\alpha_1+\alpha_2}}$ or not. The bound
applies with
$$
C_{\alpha_1,\alpha_2}=
\frac{2^{\alpha_1}}{2^{\alpha_1}-1}+\frac{2^{\alpha_2}}{2^{\alpha_2}-1}\,.
$$
\end{proof}

We next take a Littlewood-Paley decomposition of the initial data
$$
f=\sum_{j=1}^\infty\beta(2^{-2j}H)f+\beta_0(H)f\,,
$$
where $\beta(s)$ is supported in the interval $s\in[\frac 12,\frac 92]$,
and $1=\beta_0(s)+\sum_{j=0}^\infty\beta(2^{-2j}s)$ for $s\ge 0$.
Here, $H$ denotes $-\Delta$ with either Dirichlet or Neumann
conditions.
Set 
$$
f_j=e^{2^{-2j}H}\beta(2^{-2j}H)f\,,\qquad f_0=e^H\beta_0(H)f\,.
$$
By the spectral localization,
$$
\sum_{j=0}^\infty\, \|f_j\|_{H^1}^2\lesssim \|f\|_{H^1}^2\,,
$$
and we may write $u(t)=\sum_{j=0}^\infty u_j(t)\,,$
where
$$
u_j(t)=e^{-2^{-2j}H}e^{-itH}f_j\,,\qquad 
u_0(t)=e^{-H}e^{-itH}f_0\,.
$$
By the ultracontractivity estimate for $H$ on exterior domains 
(see Theorem 2.4.2 and the ensuing comments in~\cite{Dav}, where $\mu=3$
in our case), we can bound
$$
\|u_j(t)\|_{L^\infty}\lesssim 2^{\frac j3}\|e^{-itH}f_j\|_{L^9}\,.
$$
Together with the case $(p,q,s)=(3,9,\frac 12)$ of Theorem~\ref{extthm}, 
we have
$$
\|2^{-\frac j3}u_j\|_{L^3L^\infty}\lesssim
\|e^{-itH}f_j\|_{L^3 L^9}\lesssim \|f_j\|_{H^{\frac 12}}\le
2^{-\frac j2}\|f_j\|_{H^1}\,,
$$
which we combine with Minkowski's inequality to yield
$$
\biggl(\int\biggl(\,
\sum_{j=0}^\infty \|2^{\frac j6}u_j(t)\|^2_{L^\infty}
\biggr)^{\frac 32}\,dt\biggr)^{\frac 13}
\le
\biggl(\sum_{j=0}^\infty\|2^{\frac j6}u_j\|^2_{L^3L^\infty}\biggr)^{\frac 12}
\lesssim \|f\|^2_{H^1}\,.
$$
In particular,
$$
\sup_j \,2^{\frac j6}\|u_j(t)\|_{L^\infty}\le \rho_1(t)\,,
\qquad \|\rho_1\|_{L^3}\lesssim \|f\|_{H^1}\,.
$$
Similar considerations, using the case $(p,q,s)=(12,9,1)$ of 
Theorem~\ref{extthm}, yield
$$
\sup_j \,2^{-\frac j3}\|u_j(t)\|_{L^\infty}\le \rho_2(t)\,,
\qquad \|\rho_2\|_{L^{12}}\lesssim \|f\|_{H^1}\,.
$$
Lemma \ref{gagnir} now applies to give the bound
$$
\|u(t)\|_{L^\infty}\lesssim \rho_1(t)^{\frac 23}\rho_2(t)^{\frac 13}\,.
$$
Applying H\"older's inequality with the dual indices $(\frac 98,9)$ now yields
$$
\|u\|_{L^4 L^\infty}^4\lesssim\int \rho_1(t)^{\frac 83}\rho_2(t)^{\frac 43}\,dt
\lesssim \|\rho_1\|_{L^3}^{\frac 83}\,\|\rho_2\|_{L^{12}}^{\frac 43}
\lesssim \|f\|_{H^1}^4\,.
$$

%%%%%%%%%%%%%%%%%%%%%%%%%%%%%%%%%%%%%%%%%%%%%%%%%%%

\newsection{Applications to semilinear Schr\"{o}dinger equations on compact manifolds}\label{bddomain}

\noindent
In this section we consider a compact 3-dimensional
Riemannian manifold $\Omega$ with boundary. We assume
$G:[0,\infty)\rightarrow\R$ is bounded below, with $G(0)=0$, and that
\begin{equation}\label{Gcond}
|G'(r)|+r\,|G''(r)|\lesssim \langle r\rangle^{\frac 15}\,.
\end{equation}
We set $F(u)=G'(|u|^2)u$, so that
$$
|F(u)|\le\langle u\rangle^{2/5}|u|\,,\qquad |d_uF(u)|\le \langle u\rangle^{2/5}\,.
$$
We prove existence, uniqueness, and energy conservation, for initial data $u(t_0)\in H^1(\Omega)$,
to the semilinear Schr\"odinger equation
\begin{equation}\label{nls}
i\prtl_t u + \Delta u  = F(u)\,, \qquad u|_{t=t_0} =u(t_0)\,,
\end{equation}
satisfying homogeneous Dirichlet or Neumann boundary conditions \eqref{nlsbc}.
As above, by a solution to \eqref{nls} we understand that its integral form holds,
\begin{equation}\label{nlsintform}
u(t)=e^{i(t-t_0)\Delta}\biggl(u(t_0)-i\int_{t_0}^t e^{-i(s-t_0)\Delta}F(u(s))\,ds\biggr)\,.
\end{equation}
This formulation is seen to be independent
of $t_0$; that is, if $u$ solves \eqref{nlsintform} on an interval for some $t_0$
then it solves the same equation for all $t_0$ in that interval.

The key estimates we use involve values of $(p,q)$ which do not satisfy \eqref{pqcondn}. In this
case, the method of proof yields estimates with a loss of derivatives 
relative to the scale invariant value of $s$ from
\eqref{scalingcond}. In particular, the following analogue of Theorem \ref{semiclassthm} 
loses $\frac 1q$ derivatives
relative to the case of manifolds without boundary considered in \cite{burq1}. Additionally,
there are logarithmic losses due to the endpoint $p=2$ and $q=\infty$.

\begin{lemma}\label{loglosslemma}
Let $n=3$, and suppose that for all $t$, $u_\l(t,\cd)$ is spectrally localized
for $-\Delta_\g$ to the range $[\frac 14 \l^2,4\l^2]$. Then the following estimate holds,
uniformly for $6 \le q \le\infty$, where
$F_\l=(i\partial_t+\Delta_\g)u_\l$,
\begin{equation}\label{l2lqest}
\|u_\l\|_{L^2_{\l^{-1}}L^q(\Omega)}\le C
\l^{\hf-\frac 2q}(\log\l)^2\Bigl(\,\l^{\hf}\|u_\l\|_{L^2_{\l^{-1}} L^2(\Omega)}+
\l^{-\hf}\|F_\l\|_{L^2_{\l^{-1}}L^2(\Omega)}\Bigr)\,.
\end{equation}
\end{lemma}
\begin{proof}
We start by noting that the reduction of Theorem \ref{semiclassthm} to Theorem \ref{dyadicthm}
holds with uniform constant over $q\ge 6$ with $p=2$. In particular, in the handling of the normal piece, $q_0=6$ for $p=2$, and $s\le \hf$ in our estimates, so the use of Sobolev embedding works 
for that piece. Thus, \eqref{l2lqest} is a consequence of the estimate
$$
\bigl\|u_\l\|_{L^2_\epsilon L^q}
\le C\l^{1-\frac 2q}(\log\l)^2
\bigl(\|u_\l\|_{L^\infty_\epsilon L^2}+\|(i\partial_t+P_\l)u_\l\|_{L^1_\epsilon L^2}\bigr)\,,
$$
together with the following estimate, valid if $\hat{u}_\l$ is 
localized to $|\xi_n|\ge\frac 1{20}|\xi|$,
$$
\bigl\|u_\l\|_{L^2_\epsilon L^q}
\le C\l^{1-\frac 3q}(\log\l)^2
\bigl(\|u_\l\|_{L^\infty_\epsilon L^2}+\|(i\partial_t+P_\l)u_\l\|_{L^1_\epsilon L^2}\bigr)\,,
$$
where $P_\l$ is as in Theorem \ref{dyadicthm}. These estimates in turn follows as a consequence
of the following analogue of \eqref{angest}
\begin{align}
\|u_j\|_{L^2 L^q(S_{j,k})} \le C \l^{1-\frac 3q}(\log\l)^{\frac 32}
\theta_j^{\frac 12-\frac 3q}\Big(&\|u_j\|_{L^\infty L^2(S_{j,k})} +
\|F_j\|_{L^1 L^2 (S_{j,k})}\notag\\
&+ \l^{\frac 14} \theta_j^{\frac 14} \|\langle \l^{\frac 12}
\theta_j^{-\frac 12}x_n\rangle^{-1}u_j\|_{ L^2 (S_{j,k})}\label{newangest}\\
 &+ \l^{-\frac 14} \theta_j^{-\frac 14} \|\langle \l^{\frac
12} \theta_j^{-\frac 12} x_n\rangle^{2}G_j\|_{ L^2 (S_{j,k})}
\Big)\,.\notag
\end{align}
To see this, we note that for $p=2$, the branching argument \cite[p.118]{SmSoBdry} 
requires $\theta_j^\hf$ to converge, and the
remaining term $\theta_j^{-\frac 3q}$ is bounded by $\l^{\frac 1q}$.
The additional loss of $(\log\l)^{\hf}$ here comes from the fact that
there are $\sim\log\l$ terms $j$ in
the decomposition of $u_\l=\sum_j u_j$. We thus have, uniformly in $q$, 
$$
\|u_\l\|_{L^2_\epsilon L^q}\lesssim (\log\l)^\hf
\bigl\|\bigl(\sum_j|u_j|^2\bigr)^\hf\bigr\|_{L^2_\epsilon L^q}\,,
$$
and it is the norm on the right hand side that is controlled by the branching argument.

The estimate \eqref{newangest} is scale invariant; scaling by $\theta$ reduces it to the 
following analogue of \eqref{angrescale}, for angularly localized $u$
satisfying $(D_t-q(x,D))u=F+G$,
\begin{align}
\|u\|_{L^2 L^q(S)} \lesssim \mu^{1-\frac 3q}(\log\mu)^{\frac 32}
\theta^{\hf-\frac 3q}\Big(&\|u\|_{L^\infty L^2(S)} +
\|F\|_{L^1 L^2 (S)}\notag\\
&+ \mu^{\frac 14} \theta^{\frac 12} \|\langle \mu^{\frac 12} x_n\rangle^{-1}u\|_{ L^2 (S)}
\label{newangest'}\\
&+ \mu^{-\frac 14} \theta^{-\frac 12} \|\langle \mu^{\frac 12} x_n\rangle^{2}G\|_{ L^2 (S)}
\Big)\,,\notag
\end{align}
where we used that $\log\mu\approx\log\l$.

The reduction of \eqref{newangest'} to homogeneous estimates, that is, bounds on
the operator $W$ of \eqref {West}, involves a loss of $\log\mu$ due to the fact that $p=2$.
This comes from the
use of the $V^2_q$ spaces introduced by Koch and Tataru~\cite{KT}, where the subscript $q$
refers to the Hamiltonian flow for $q(x,\xi)$.
In case $p=2$,
one needs to control the 2-atomic norm $U^2_q$ of $\tilde u$, whereas 
$V^2_q\subset U^p_q$ only for $p>2$. To proceed, we note that in the atomic
decomposition argument of \cite[Lemma 6.4]{KT}, we may truncate the sum
$u=\sum_n v_n$ to $n\lesssim \log\mu$, since the error is bounded in $L^\infty L^2$ by
$\mu^{-N}$, and its contribution 
thus may be estimated in the desired norm using Sobolev embedding. Each term
$v_n$ is uniformly bounded in $U^2_q$, hence the $U^2_q$ norm of the truncated sum
is $\lesssim\log\mu$.

We are thus reduced to establishing the following analogue of \eqref{West},
\begin{equation}\label{West'}
\|W\! f\|_{L^2 L^q (S)} \le C \mu^{1-\frac 3q} (\log\mu)^\hf\theta^{\hf-\frac 3q}
\|f\|_{L^2(\RR^{2n})}\,.
\end{equation}

To establish \eqref{West'}, we consider $W\!W^*$ as in the proof of Theorem \ref{coniclocthm}.
Taking $n=3$ in \eqref{qq'est}, we 
note the following integral bound for $6\le q\le \infty$, $\mu$ and $\theta$ as above,
$$
\int_0^1(\mu^{-1}+t)^{-(1-\frac 2q)}(\mu^{-1}\theta^{-2}+t)^{-\hf(1-\frac 2q)}\,dt\le
C\mu^{\frac 32(1-\frac 2q)-1}(\log\mu)\theta^{1-\frac 6q}\,,
$$
where $C$ is uniformly bounded. The estimate \eqref{West'} follows by Schur's lemma.
\end{proof}

We use Lemma \ref{loglosslemma}
to deduce the following analogue of Lemma 3.6 of \cite{burq1}. This version is weaker, both in the logarithmic loss and the loss of $\l^{\frac 1q}$, but is sufficient for
our purposes. From now on, we let $u_\l=\beta(\l^{-2}H)u$ denote a Littlewood-Paley decomposition
of $u$, where $\l=2^k$ and $k\ge 1$. The term $k=0$ contains the low frequency terms of
$u$, and the bounds for this term will follow similarly to $k=1$.

\begin{lemma} Let $u$ solve \eqref{nlsintform}. Then there are $C<\infty$ and $\epsilon>0$ such that,
uniformly for $6\le q \le\infty$,
the following holds on any time interval $[0,T]$ with $\l^{-1}\le T\le 1$, 
\begin{equation}\label{lamterm}
\|u_\l\|_{L^2([0,T],L^q)}\le C\l^{-\frac 2q}(\log\l)^2
\Bigl(\|u_\l\|_{L^2([0,T],H^1)}+
\l^{-\epsilon}\bigl\langle\|u\|_{L^\infty([0,T],H^1)}\bigr\rangle^{7/5}\Bigr)\,.
\end{equation}
\end{lemma}
\begin{proof}
We divide $[0,T]$ into subintervals of length $\l^{-1}$. We apply \eqref{l2lqest} on each such subinterval, and square sum over subintervals to obtain
$$
\|u_\l\|_{L^2([0,T],L^q)}\le 
C\,\l^{-\frac 2q}(\log\l)^2
\Bigl(\|u_\l\|_{L^2([0,T],H^1)}+\|F_\l\|_{L^2([0,T],L^2)}\Bigr)\,,
$$
where $F_\l=F(u)_\l$.
We now take
$$
\alpha=\frac 25\,,\qquad r=\frac{6}{3+\alpha}\,,\qquad \epsilon=3\Bigl(\frac 1r-\frac 12\Bigr)-1\,,
$$
and observe that
\begin{align*}
\|F_\l\|_{L^\infty L^2}
&\lesssim \l^{-\epsilon}\|F\|_{L^\infty W^{1,r}}\\
&\lesssim \l^{-\epsilon}\|\langle u\rangle^\alpha(|d_x u|+\langle u\rangle)\|_{L^\infty L^r}\\
&\lesssim \l^{-\epsilon}\|\langle u\rangle\|_{L^\infty L^6}^\alpha
\|(|d_x u|+\langle u\rangle)\|_{L^\infty L^2}\\
&\lesssim \l^{-\epsilon}\bigl(1+\|u\|_{L^\infty H^1}\bigr)^{\alpha+1}\,.
\end{align*}
\end{proof}

Sobolev embedding yields 
$\|u_{<T^{-1}}\|_{L^2_T L^q}\lesssim T^{-\frac 12}\|u\|_{L^2_T H^1}\lesssim \|u\|_{L^\infty_T H^1}$,
where $u_{<T^{-1}}$ denotes the sum of $u_\l$ over $\l<T^{-1}$.
Summing \eqref{lamterm} over $\l=2^{-k}$, and using Cauchy-Schwarz over $k$, we conclude that, 
with $C$ uniform over $q\ge 6$,
\begin{equation}\label{l2lqbound}
\|u\|_{L^2([0,T],L^q)}\le C\,q^{\frac 52}\,\bigl(1+\|u\|_{L^\infty H^1}\bigr)^{7/5}\,.
\end{equation}

Suppose now that $u$ satisfies \eqref{nlsintform} on a time interval $[0,T]$, 
where $u(t_0)\in H^1(\Omega)$. For sufficiently regular solutions $u$, we have the conservation laws
\begin{equation}\label{conservlaw}
\begin{split}
\int_\Omega |u(t)|^2&=\int_\Omega |u(t_0)|^2\\
\int_\Omega |d_x u(t)|_\g^2+G(|u(t)|^2) &=\int_\Omega |d_x u(t_0)|_\g^2+G(|u(t_0)|^2)
\end{split}
\end{equation}
In particular, since $-C\le G(r)\le C\langle r\rangle^{\frac 65}$, it follows that
$\|u\|_{L^\infty([0,T],H^1)}\lesssim 1+\|u(t_0)\|_{H^1}$, uniformly in $T$.

In the following proof, we assume a priori that $u\in L^\infty H^1$ and prove uniqueness
of such solutions.
The existence of bounded energy solutions, and energy conservation,
is then proved by a weak-limit argument.

\begin{theorem}
For each data $f\in H^1(\Omega)$, and all $T>0$, there exists a unique solution $u$ to
the equation \eqref{nlsintform},
subject to the condition $u\in L^\infty([0,T],H^1(\Omega))$.
Furthermore, the solution satisfies the conservation laws \eqref{conservlaw}.
\end{theorem}
\begin{proof}
We start with the uniqueness of solutions. Since $u\in L^\infty([0,T],H^1)$
it follows by Sobolev embedding that $F(u)\in L^\infty L^2$, so $u\in C([0,T],L^2)$, and
by interpolation $u\in C([0,T],H^s)$ for all $s<1$. Repeating this argument shows that the term
in parentheses in \eqref{nlsintform} belongs to $C^1([0,T],L^2)$.

Let $u$ and $v$ be two solutions to \eqref{nlsintform}, with $u(0)=v(0)$.
By unitarity of $\exp(it\Delta)$, 
\begin{align*}
\frac d{dt}\|u(t)-v(t)\|_{L^2}^2
&=\frac d{dt}\bigl\|e^{-it\Delta}\bigl(u(t)-v(t)\bigr)\bigr\|_{L^2}^2\\
&=2\,\Im\bigl\langle F(u(t))-F(v(t)),u(t)-v(t) \bigr\rangle\\
&\le C\int_\Omega \bigl(\langle u(t)\rangle^{2/5}+\langle v(t)\rangle^{2/5}\bigr)|u(t)-v(t)|^2\\
&\le C\bigl(1+\|u(t)\|_{L^{2q/5}}+\|v(t)\|_{L^{2q/5}}\bigr)^{\frac 25}\|u(t)-v(t)\|_{L^{2q'}}^2
\end{align*}
provided $q\ge 5/2$.
Since $\|u(t)-v(t)\|_{L^6}\le C$, we may interpolate to bound
$$
\|u(t)-v(t)\|_{L^{2q'}}\le C\|u(t)-v(t)\|_{L^2}^{1-\frac 3{2q}}\,.
$$
Setting $g(t)=\|u(t)-v(t)\|_{L^2}^2$, and noting $g(0)=0$, we have upon integrating that
$$
g(\tau)^{\frac 3{2q}}\le 
\frac Cq\int_0^\tau\Bigl(\|u(t)\|_{L^{2q/5}}^{\frac 25}+\|v(t)\|_{L^{2q/5}}^{\frac 25}\Bigr)\,dt
+\frac{C\tau}{q}\,.
$$
By H\"older's inequality and \eqref{l2lqbound}, for $\tau\in[0,T]$
$$
\int_0^\tau \|u(t)\|_{L^{2q/5}}^{\frac 25}\,dt\le 
\tau^{\frac 45}\|u\|_{L^2([0,T],L^{2q/5})}^{\frac 25}\le C\tau^{\frac 45}q\,.
$$
Consequently,
$$
g(\tau)\le \biggl(C\tau^{\frac 45}+\frac {C\tau}q\biggr)^\frac{2q}{3}
$$
which goes to $0$ as $q\rightarrow\infty$, provided $\tau$ is small depending on $C$. Repeating
the argument yields uniqueness on $[0,T]$.

To establish existence and energy conservation for
\eqref{nlsintform} with $H^1$ data, we 
let $G_j(r)$ be a family of smooth, compactly supported real valued functions on 
$[0,\infty)$, uniformly bounded below, such that $G_j(r)$ and $G_j'(r)$ converge uniformly
on compact sets to $G(r)$ and $G'(r)$. Additionally, we require that \eqref{Gcond} holds uniformly over $j$ for $G=G_j$.

We fix a time $t_0$ and initial data $u(t_0)\in H^1(\Omega)$, and let $u_j(t)$ solve \eqref{nlsintform} where $F$ is replaced by $F_j=G_j'(|u|^2)u$. We assume for the moment that $u_j$ exists globally in time, and satisfies the conservation law \eqref{conservlaw}, with $G$ replaced by $G_j$.
In particular 
$$
\|u_j\|_{L^\infty H^1}\le C\bigl(1+\|u(t_0)\|_{H^1}\bigr)\quad \text{uniformly over} \; j\,.
$$
By \eqref{nlsintform}, $\exp(-i(t-t_0)\Delta)u_j(t)$ is uniformly bounded in 
$C^1L^2\cap L^\infty H^1\subset C^\hf H^\hf$, hence by the theorems
of Rellich and Arzela-Ascoli, some subsequence of $u_j$ converges uniformly in the $L^2$ norm
on each finite time interval to $u(t)$, in the sense that
$$
\lim_{n\rightarrow\infty}\|u_{j(n)}-u\|_{C([-T,T],L^2)}=0\quad\text{for all}\;\,T<\infty\,.
$$
It follows that $u\in L^\infty H^1$, and thus
by interpolation that for all $s<1$,
$$
\lim_{n\rightarrow\infty}\|u_{j(n)}-u\|_{C([-T,T],H^s)}=0\quad\text{for all}\;\,T<\infty\,.
$$
By Sobolev embedding we deduce that $F_j(u_j)\rightarrow F(u)$ in $C([-T,T],L^2)$, hence
$u$ is the solution to \eqref{nlsintform}, unique by above. The conservation of mass in
\eqref{conservlaw} follows by uniform convergence in the $L^2$ norm and conservation of mass
for $u_j$. To conclude, we observe that by energy conservation for $u_j$ and Fatou's lemma, for each $t_1$ we have
\begin{equation}\label{energyineq}
\int_\Omega |d_x u(t_1)|_\g^2+G(|u(t_1)|^2) \le \int_\Omega |d_x u(t_0)|_\g^2+G(|u(t_0)|^2)\,.
\end{equation}
On the other hand, $u$ is the unique solution with data $u(t_1)$ at time $t_1$, and
the inequality is thus symmetric under exchange of $t_0$ and $t_1$.

It remains to prove existence of energy conserving solutions to \eqref{nlsintform} for
$H^1$ data, in case $G(r)\in C^\infty_c(\R)$. For convenience set $t_0=0$.
We introduce $w(t)=\exp(-it\Delta)u(t)$, and write \eqref{nlsintform} as
\begin{equation}\label{nlsintformw}
w(t)=u_0-i\int_0^t e^{-is\Delta}F\bigl(e^{is\Delta} w(s)\bigr)\,ds\,.
\end{equation}

Since $F(z)=G'(|z|^2)z\in C^\infty_c(\CC)$,
the map $u\rightarrow F(u)$ is globally Lipschitz on $L^2(\Omega)$,
and one has existence, uniqueness, and Lipschitz dependence on initial data for 
$C^1L^2$ solutions of \eqref{nlsintformw}, given by the
limit of $w_n(t)$, where $w_0(t)=u_0$, and
\begin{equation}\label{nlsintformwn}
w_{n+1}(t)=u_0-i\int_0^t e^{-is\Delta}F\bigl(e^{is\Delta} w_n(s)\bigr)\,ds\,.
\end{equation}
Convergence of $w_n$ to $w$ is uniform in the $L^2$ norm on any compact interval. From unitarity
of $\exp(it\Delta)$ on $H^k$ (with norm defined spectrally), and the bound
$$
\|F(w(s))\|_{H^1}\le K\|w(s)\|_{H^1}\,,
$$
one sees from \eqref{nlsintformwn} and weak limits, and using \eqref{nlsintformw} to express $w'(t)$,
that 
\begin{equation}\label{wH1bounds}
\|w(t)\|_{H^1}\le \|u_0\|_{H^1}\exp(Kt)\,,\qquad \|w'(t)\|_{H^1}\le K\|u_0\|_{H^1}\exp(Kt)\,.
\end{equation}
It remains to prove the conservation laws \eqref{conservlaw} on an interval $[0,T]$, for a
$T$ depending only on $\|u_0\|_{H^1}$ and $F$; uniqueness yields global conservation. To do this,
we will prove for such a $T$ that if $u_0\in H^2$, then
$w\in C^1([0,T],H^2)$, hence $u\in C^1([0,T],L^2)$. Together this is sufficient regularity to
see that \eqref{conservlaw} holds on $[0,T]$ for $u_0\in H^2$. 
Density and Fatou's lemma yields mass conservation and
\eqref{energyineq} for $H^1$ data; uniqueness then yields \eqref{conservlaw}.

We start by noting that
$$
\|F(u(s))\|_{H^2}\lesssim \|u(s)\|_{W^{1,4}}^2+\|u(s)\|_{H^2}\lesssim \|u(s)\|_{H^2}^2+\|u(s)\|_{H^2}\,.
$$
Iterating \eqref{nlsintformwn} yields $\|u\|_{L^\infty([0,T'],H^2)}\le 2\|u_0\|_{H^2}$ for some $T'>0$ depending
on $\|u_0\|_{H^2}$.
It suffices then to
prove, for some $C$ and $T$ depending only on $\|u_0\|_{H^1}$, that if $T'\le T$ and
$\|u\|_{L^\infty([0,T'],H^2)}<\infty$, then $\|u\|_{L^\infty([0,T'],H^2)}\le C\|u_0\|_{H^2}$.
Theorem \ref{intthm} and \eqref{nlsintform} yield
\begin{align*}
\|u\|^2_{L^4([0,T'],W^{1,4})}
&\lesssim \|u_0\|^2_{H^{3/2}}+\Bigl(\int_0^{T'}\|F(u(s))\|_{H^{3/2}}\,ds\Bigr)^2\\
&\lesssim \|u_0\|_{H^1}\|u_0\|_{H^2}+\int_0^{T'}\|F(u(s))\|_{H^1}\|F(u(s))\|_{H^2}\,ds\,,
\end{align*}
and we can also use \eqref{nlsintform} to bound 
$\|u\|_{L^\infty_{T'}H^2}\le\|u_0\|_{H^2}+\|F(u)\|_{L^1_{T'}H^2}$.
By the bounds \eqref{wH1bounds}, we combine these estimates, assuming $T'\le T\le 1$, to yield
$$
\|u\|_{L^\infty_{T'}H^2}+\|F(u)\|_{L^2_{T'}H^2}\le 
C\|u_0\|_{H^2}+CT^{\hf}\|u\|_{L^\infty_{T'}H^2}+CT^{\hf}\|F(u)\|_{L^2_{T'}H^2}\,,
$$
where $C\lesssim\|u_0\|_{H^1}$. Taking $T$ small yields the desired result.
\end{proof}
We conclude by noting that the above argument shows that $u\in C([0,T],H^2)$ for all finite
$T$ if $u_0\in H^2$, but possibly with exponential growth of the $H^2$ norm, with the growth
constant depending on $\|u_0\|_{H^1}$.

\bigskip

\noindent\textbf{Acknowledgements.}
The authors would like to thank the referee for suggesting the application of our methods to semilinear Schr\"odinger equations on compact manifolds in Section \ref{bddomain}.

%%%%%%%%%%%%%%%%%%%%%%%%%%%%%%%%%%%%%%%%%%%%%%%%%%%

\end{document}